\documentclass[a4paper,11pt]{amsart}

\usepackage[hang,flushmargin]{footmisc}
\usepackage{marginnote}
\usepackage{graphicx,rotating}
\usepackage[font=footnotesize,labelfont=footnotesize]{caption}
\usepackage[utf8]{inputenc}
\usepackage[toc]{appendix}
\usepackage{twoopt}
\usepackage{float}
\usepackage{subfig}
\usepackage{orcidlink}
\allowdisplaybreaks
\usepackage{amssymb}
\usepackage{amsthm}
\usepackage{amsmath,a4wide}
\numberwithin{equation}{section}
\usepackage{hyperref}
\usepackage{mathtools}
\usepackage{amsfonts}
\usepackage[foot]{amsaddr}

\usepackage{marvosym}
\usepackage{verbatim}
\usepackage{theoremref}

\usepackage{enumerate}
\usepackage{pdfsync}
\usepackage{caption}
\usepackage{hyperref}
\usepackage{enumerate}
\mathtoolsset{showonlyrefs}

\theoremstyle{plain}
\newtheorem{theorem}{Theorem}[section]
\theoremstyle{plain}
\newtheorem{corollary}[theorem]{Corollary}
\theoremstyle{plain}
\newtheorem{lemma}[theorem]{Lemma}
\theoremstyle{plain}
\newtheorem{proposition}[theorem]{Proposition}
\theoremstyle{definition}
\newtheorem{definition}[theorem]{Definition}
\theoremstyle{remark}
\newtheorem*{remark}{Remark}
\theoremstyle{remark}

\theoremstyle{remark}

\theoremstyle{definition}
\newtheorem{example}{Example}[section]
	
\newtheorem{proposition*}[example]{Proposition}
\theoremstyle{plain}
\theoremstyle{definition}
\theoremstyle{plain}
\newtheorem*{theorem*}{Theorem}

\usepackage{hyperref}
\hypersetup{colorlinks=false, linktoc=all, linkcolor=black}
\usepackage{tikz}
\usepackage{tikz-cd}

\newcommand\rurl[1]{%
  \href{https://#1}{\nolinkurl{#1}}%
}
\newcommand{\N}{\mathbb{N}}

\newcommand{\R}{\mathbb{R}}
\newcommand{\C}{\mathbb{C}}
\renewcommand{\H}{\mathbb{H}}
\newcommand{\J}{\mathcal{J}}

\newcommand{\ST}{\mathcal{D}_{S}}
\newcommand{\SI}{S^{-1}}
\newcommand{\F}{\mathcal{F}}
\newcommand{\w}[1]{\mathcal{D}_{#1}}
\newcommand{\h}[1]{\widehat{#1}}

\newcommand{\Mp}{\mathrm{Mp}(2d,\R)}
\newcommand{\Sp}{\mathrm{Sp}(2d,\R)}
\newcommand{\AF}{\mathrm{A}}
\newcommand{\Tp}{\mathrm{Tp}}
\newcommand{\MSpace}[2]{\mathrm{M}^{#1}_{#2}(\R^d)}
\newcommand{\norm}[1]{\left\lVert#1\right\rVert}
\newcommand{\abs}[1]{\left\lvert#1\right\rvert}
\numberwithin{equation}{section}
\newcommand\myeq[2]{\overset{\scriptscriptstyle #1}{#2}}
\DeclareMathOperator*{\essinf}{ess\,inf}
\DeclareMathOperator*{\esssup}{ess\,sup}
\DeclareMathOperator{\dom}{dom}

\DeclareMathOperator{\diag}{diag}

\newcommand{\dotp}{
    \mathop{
        \mathchoice{\vcenter{\hbox{\Huge$\cdot$}}}
                   {\vcenter{\hbox{\Huge$\cdot$}}}
                   {\vcenter{\hbox{\Large$\cdot$}}}
                   {\vcenter{\hbox{\small$\cdot$}}}
    }
}

\usepackage{xcolor}

\usepackage{pifont}
\makeatletter
\newcommand\thankssymb[1]{\textsuperscript{\@fnsymbol{#1}}}

\begin{document}

\title[The metaplectic action on modulation spaces]{The metaplectic action on modulation spaces}

\author[H.\ F\"uhr]{Hartmut Führ \orcidlink{0000-0001-5718-6449
}\thankssymb{1}}
\email{fuehr@mathga.rwth-aachen.de}

\author[I.\ Shafkulovska]{Irina Shafkulovska \orcidlink{0000-0003-1675-3122} \thankssymb{2}\textsuperscript{,}\thankssymb{3}}
\email{irina.shafkulovska@univie.ac.at}

\address{\thankssymb{1}Lehrstuhl f\"ur Geometrie und Analysis, RWTH Aachen University, D-52056 Aachen, Germany}
\address{\thankssymb{2}Faculty of Mathematics, University of Vienna, Oskar-Morgenstern-Platz 1, 1090 Vienna, Austria}
\address{\thankssymb{3}Acoustics Research Institute, Austrian Academy of Sciences, Wohllebengasse 12-14, 1040 Vienna, Austria}
\thanks{I. Shafkulovska was supported by the Austrian Science Fund (FWF) project P33217 and thanks the Austrian Academy of Sciences for hosting an internship at the Acoustics Research Institute.}

\begin{abstract}
    We study the mapping properties of metaplectic operators $\widehat{S}\in \mathrm{Mp}(2d,\mathbb{\R})$ on modulation spaces of the type  $\mathrm{M}^{p,q}_m(\mathbb{\R}^d)$. 
    Our main result is a full characterisation of the pairs $(\widehat{S},\mathrm{M}^{p,q}{}(\mathbb{R}^d))$ for which the operator  $\widehat{S}:\mathrm{M}^{p,q}{}(\mathbb{\R}^d) \to \mathrm{M}^{p,q}{}(\mathbb{\R}^d)$ is \textit{(i)} well-defined, \textit{(ii)}~bounded. It turns out that these two properties are equivalent, and they entail that $\h{S}$ is a Banach space automorphism. For polynomially bounded weight functions, we provide a simple sufficient criterion 
    to determine whether the well-definedness (boundedness) of ${\widehat{S}:\mathrm{M}^{p,q}{}(\mathbb{\R}^d)\to \mathrm{M}^{p,q}(\mathbb{\R}^d)}$ transfers to $\widehat{S}:\mathrm{M}^{p,q}_m(\mathbb{\R}^d)\to \mathrm{M}^{p,q}_m(\mathbb{\R}^d)$.
    
\vspace{2em}


\end{abstract}

\maketitle

\section{Introduction}
This paper studies the interaction of two fundamental objects (or classes of objects) in time-frequency analysis: the metaplectic group on the one hand, and the scale of modulation spaces $\MSpace{p,q}{m}$ on the other. The metaplectic group can be understood as the fundamental symmetry group in time-frequency analysis; see \cite{Folland1989, Groechenig2001} for more details concerning the metaplectic group and its role in time-frequency analysis and the representation theory of the Heisenberg group. The group is often employed to reduce the study of general cases to more specific, concrete settings; recent applications of this principle can be found in \cite{grohsLiehr4, GrohsLiehr2022, NicolaTrapasso2020} as recent illustrations of this principle. 
Modulation spaces, on the other hand, are spaces consisting of signals (distributions) whose windowed Fourier transform exhibits a certain decay, as quantified by a weighted mixed $L^p$-norm. We refer to \cite{Cordero_Rodino, Groechenig2001} for more background on modulation spaces. 

We are interested in \textit{invariance properties}, asking under which conditions an operator $\widehat{S}$ belonging to the metaplectic group, initially defined as a unitary operator on $L^2(\R^d)$, extends to a bounded operator $\MSpace{p,q}{m} \to \MSpace{p,q}{m}$. The question was first raised in \cite{CauliEtAl2019}, where the invariance of $\MSpace{p}{}$ under arbitrary metaplectic operators was instrumental in establishing generalized Strichartz estimates. For a subclass of metaplectic operators which are also Fourier integral operators, sufficient conditions and a counterexample hinting at necessary conditions were already established in \cite[Sec. 7]{CorderoEtAl2010_2}. Given the relevance both of metaplectic operators and modulation spaces to the field of time-frequency analysis, understanding their interplay can be regarded as a fundamental question of independent interest. As the subsequent results and proofs show, the question turns out to be rather more nuanced than what the already understood special case $p=q$ and $m\equiv 1$, considered in \cite{CauliEtAl2019}, suggests. 

\subsection{Contributions}
In this subsection, we provide a brief overview of the main results of the article. For definitions of the various objects in the following theorems, we refer to the subsequent sections and the background material contained in the textbooks \cite{Cordero_Rodino, Groechenig2001}. 

Recall that the metaplectic group is the double cover of the symplectic group, i.e., there exists a surjective Lie group homomorphism
$\pi^{\mathrm{Mp}}:\Mp\to\Sp$.
Our main new result, restated below as \thref{thm:main_unweighted}, is a characterisation of triples $(p,q,\widehat{S})$ of exponents $p,\,q\in[1,\infty]$ and metaplectic operators $\widehat{S}\in\Mp$ such that 
\begin{equation}
    \widehat{S}:\MSpace{p,q}{}\to\MSpace{p,q}{}
\end{equation}
is well-defined.
\begin{theorem*}
Let $p,q\in [1,\infty]$ and $\widehat{S}\in \Mp$ be given. The following statements are equivalent:
\begin{enumerate}
\item $\widehat{S}:\MSpace{p,q}{} \to \MSpace{p,q}{} $ is well-defined.
\item $\widehat{S}:\MSpace{p,q}{} \to \MSpace{p,q}{} $ is well-defined and bounded.
\item One of the following conditions holds:
\begin{enumerate}[(i)]
\item $p=q$, or
\item $p\neq q$ and the projection $\pi^{\mathrm{Mp}}(\widehat{S})=S\in \Sp $ is an upper block triangular matrix.
\end{enumerate}
\end{enumerate}
\end{theorem*}
This can be lifted to weighted modulation spaces if the metaplectic operator is compatible with the weight function (see \thref{thm:main_weighted}).
\begin{theorem*}
Let $m$ be a moderate polynomially bounded weight, $p,q\in[1,\infty]$. If $\widehat{S}\in \Mp$ with projection $\pi^{\tiny Mp}(\widehat{S})=S$ satisfies $m\asymp m\circ S^{-1}$, then $\widehat{S}$ is a bounded operator from $\MSpace{p,q}{} $ to $\MSpace{p,q}{} $ if and only if it is a bounded operator from $\MSpace{p,q}{m} $ to $\MSpace{p,q}{m} $.
\end{theorem*}

\subsection{Outline}
The paper is structured as follows. The \hyperref[sec:notation]{second section} covers all necessary notions related to modulation spaces and metaplectic operators. The \hyperref[sec:unweighted]{third section} contains the proof of the classification for unweighted modulation spaces. The \hyperref[sec:weighted]{last section} is dedicated to lifting the result from the unweighted to the weighted setting. Two appendices contain somewhat lengthy computations of certain ambiguity functions and their dilates. 

\section{Basic definitions and notation}\label{sec:notation} We begin with a short introduction to time-frequency analysis. The \textit{translation} and the \textit{modulation operator} (\textit{time shift and frequency shift}) are denoted with  $T_x f(t) = f(t-x)$ and $M_\omega f(t)=e^{2\pi i \omega\cdot t}f(t)$, respectively. We will make use of the unitary Fourier transform
\begin{equation}
\F f(\omega)=\int_{\R^{d}} f(x)e^{-2\pi i \omega\cdot x} dx,\qquad \omega\in\R^{d},
\end{equation}
pointwise defined on $L^1(\R^d)$ and continuously extended by duality on $\mathcal{S}'(\R^{d})$. 
Both types of shifts are unitary representations of the abelian group $\R^d$, so they commute within their class of shifts. However, this does not transfer to their combinations, as we obtain a non-trivial phase factor
\begin{equation}\label{eq:MT=eTM}
    M_\omega T_x = e^{2\pi i \omega \cdot x} T_x M_\omega, \qquad x,\omega\in\R^d.
\end{equation}
A recurring theme in this paper are intertwining properties. The Fourier transform can be seen as the elementary example with $\F T_x = M_{-x}\F$, $\F M_\omega = T_\omega \F$, which justifies the name \textit{frequency shift} for the modulations. 

The cross-ambiguity function associated with $f,g\in L^2(\R^d)$ is the map
\begin{equation}
\begin{split}
 \AF(f,g)(x,\omega) &=\int_{\R^d} f\left(t+\frac{x}{2}\right) \overline{g\left(t-\frac{x}{2}\right)} e^{-2\pi i \omega\cdot t}\, dt= \left\langle f, T_{x/2}M_\omega T_{x/2} g\right\rangle.
\end{split}
\end{equation}
The latter equality allows us to extend the definition of the cross-ambiguity function for any $f\in \mathcal{S}'(\R^d)$, $g\in  \mathcal{S}(\R^d)$. It corresponds to the symmetric time-frequency shift 
\begin{equation}
\rho(\lambda) = T_{x/2}M_\omega T_{x/2} = M_{\omega/2}T_x M_{\omega/2},\qquad \lambda = \left(x\, ,\,\omega\right)^t,
\end{equation} 
but each parameterization with $\tau\in[0,1]$
\begin{equation}
    \rho_\tau(x,\omega)  \coloneqq T_{\tau x} M_{\omega} T_{(1-\tau) x} = e^{2\pi i (1-2\tau)\omega\cdot x} M_{\tau \omega} T_{x} M_{(1-\tau)\omega}
\end{equation}
induces a sesqui-linear time-frequency representation
\begin{equation}\label{eq:tau-representation}
    \left\langle f, \rho_\tau (x,\omega) g \right\rangle =e^{2\pi i \tau \omega\cdot x - \pi i \omega\cdot x}\left\langle f, \rho_{1/2} (x,\omega) g \right\rangle = e^{2\pi i \tau \omega\cdot x - \pi i \omega\cdot x} \AF(f,g)(x,\omega).
\end{equation}
Taking $\tau = 1/2$ results in the symmetric time-frequency shift. As another prominent choice, the completely polarized version $\tau =0$ reproduces the \textit{short-time Fourier transform} (STFT) 
\begin{equation}
    V_g f(x,\omega ) =\int_{\R^d} f\left(t\right) \overline{g\left(t-x\right)} e^{-2\pi i \omega\cdot t}\, dt = \left\langle f, M_\omega T_{x} g\right\rangle.
\end{equation}
A further closely related transform is the \emph{cross-Wigner distribution} 
$$
\mathrm{W}(f,g)(x,\omega) = \int_{\R^d} f\left(x+\tfrac{t}{2}\right) \overline{g\left(x-\tfrac{t}{2}\right)} e^{-2\pi i \omega\cdot t}\, dt = 2^d \AF(f, g(-\dotp))(2x,2\omega),
$$
which is often studied in the context of quantum mechanics \cite{Gosson2006}.

The relation \eqref{eq:MT=eTM} does not allow for $\rho_\tau(\lambda+\nu) =\rho_\tau(\lambda)\rho_\tau(\nu)$ to hold for arbitrary $\lambda,\,\nu\in\R^{2d}$, but the appearing phase factor induces a corresponding Heisenberg group $\H_\tau$. Topologically, the Heisenberg group can be identified with $\R^d\times\R^d\times \R$. However, this does not hold on the group level, as the multiplication is non-commuting, chosen in such a way that $(\lambda, t)\mapsto e^{2\pi i t}\rho_\tau(\lambda)$ is a unitary representation on the Hilbert space $L^2(\R^d)$. This allows for methods from representation theory and gives rise to a non-trivial family of special intertwining operators, the metaplectic operators, which will later be addressed in depth. For further exploration of these parameterized representations, we refer the reader to \cite{BoggiattoEtAl2010}, \cite{BoggiattoEtAl2010_2} and \cite{Sandikci2020}.

Throughout this paper, $\norm{\dotp}_{p}\coloneqq\norm{\dotp}_{L^p(\R^d)}$, denotes the $p$-norm on $L^p(\R^d)$, $p\in[1,\infty]$. We consider modulation spaces associated to mixed-norm weighted $L^p$-spaces, more formally defined as the spaces $L^{p,q}_m(\R^{2d})$ of measurable functions $f:\R^d\times \R^d \to \C$ with the norm
\begin{equation*}
\norm{f}_{p,q,m}\coloneqq \norm{\omega\mapsto\norm{f(\dotp, \omega)\, m(\dotp, \omega)}_{p}}_{q}.
\end{equation*}
Here $m:\R^{2d}\to (0,\infty)$ is a weight that is \textit{$v$-moderate},  i.e., there is a $C>0$ with
\begin{equation*}
    m(x + y)\leq C\, v(x)\cdot m(y),\qquad x,y,\in\R^{2d}
\end{equation*}
for some weight $v$. Throughout this paper we reserve the term \textit{moderate weight} to even weights whose controlling weight can be chosen as $v(x) = (1+\|x\|)^N$, for a suitable exponent $N>0$. Note that this implies  $m(x) \le C' (1+\|x\|)^N$, i.e., a moderate weight $m$ is itself polynomially bounded. The control weight $v$ is  \textit{submultiplicative}, which amounts to the estimate 
\begin{equation}
       v(x + y)\leq C\, v(x)\cdot v(y),\qquad x,y,\in\R^{2d}.
\end{equation}
 As a consequence of submultiplicativity, $L^{p,q}_m(\R^{2d})$ turns out to be invariant under translations \cite[Proposition 11.1.2.]{Groechenig2001}, which is fundamental for the well-definedness of the associated modulation spaces.

Mixed-norm Lebesgue spaces were systematically explored in \cite{BenedekPanzone1961}. Most properties, such as completeness, 
 duality and interpolation relations hold in an analogous fashion to the versions of the standard Lebesgue spaces. 
The modulation spaces associated to $L^{p,q}_m(\R^{2d})$ are then given by
\begin{equation*}
    \MSpace{p,q}{m} = \left\lbrace f\in \mathcal{S}'(\R^d)\mid \AF(f,g)\in L^{p,q}_m(\R^{2d})\right\rbrace,
\end{equation*}
where $g\in\mathcal{S}(\R^d) \setminus \{0 \}$ is fixed. It is well-known that $\MSpace{p,q}{m}$ is independent of $g$, and that the associated norm 
\begin{equation*}
    \norm{f}_{\MSpace{p,q}{m}}\coloneqq \norm{\AF(f,g)}_{p,q,m}
\end{equation*}
is independent up to equivalence. Traditionally, the modulation spaces are defined using the short-time Fourier transform, but the phase factor in \eqref{eq:tau-representation} has no impact on the definition or the induced norms. Due to the algebraic relation to the ambiguity function, the cross-Wigner distribution can be used as well. Using the ambiguity function instead of the windowed Fourier transform will make the interaction with the metaplectic action somewhat easier, specifically in the connection with twisted convolution in the proof of \thref{thm:main_weighted}. In the spirit of the STFT, we will refer to $g\in\mathcal{S}(\R^d)$ as \textit{the window function} of $\AF(\dotp,g)$. 

Modulation spaces are Banach spaces, with natural, continuous embeddings as closed subspaces of $L^{p,q}_m(\R^{2d})$. The reconstruction formula is given by
\begin{equation}
   \left\langle g,\gamma\right\rangle f  = \AF(\dotp, g)^*\, \AF(f,\gamma),
\end{equation}
where the adjoint acts (in the weak sense) as
\begin{equation}
   \AF(\dotp,g)^*F  = \int_{\R^{2d}} F(x,\omega)\, T_{x/2}M_{\omega}T_{x/2}g\ d(x,\omega).
\end{equation}
The following result was first noted in \cite{Feichtinger1983}. See also \cite[Section 2.3]{Cordero_Rodino}. 
\begin{proposition}
\thlabel{prop:inclusions}
Let $p,q, r,s \in[1,\infty]$. The following relations hold:
\begin{enumerate}[(a)]
\item $\MSpace{p,q}{m}\subseteq \MSpace{r,s}{m}$ if and only if $p\leq r$ and $q\leq s$.
\item  $(\MSpace{p,q}{m} )' =\MSpace{p',q'}{1/m}$, if $1\leq p,q<\infty$, with 
\begin{equation}
    \abs{\left\langle f,g\right\rangle} \lesssim \norm{f}_{\MSpace{p,q}{m}}\norm{g}_{\MSpace{p',q'}{1/m}},
\end{equation}
where $p'$ and $q'$ are the H\"older conjugates of $p$ and $q$, respectively.
\item If the weight function is polynomially bounded, 
\begin{equation}\label{eq:inclusions}
\mathcal{S}(\R^{d})\subseteq \MSpace{1}{m}\subseteq \MSpace{p,q}{m}\subseteq \MSpace{\infty}{m}\subseteq \mathcal{S}'(\R^d).
\end{equation}
\end{enumerate}
\end{proposition}
For a more detailed overview of modulation spaces, we refer to \cite{Cordero_Rodino, Gosson2011, Grafakos2014, Groechenig2001}. 

\subsection{The symplectic group and the metaplectic group}
Both time and frequency shifts are representations of the abelian group $\R^d$, so they commute within their class, but symmetric time-frequency shifts do not satisfy $\rho(\lambda)\rho(\nu) = \rho(\lambda+\nu)$ for arbitrary $\lambda, \nu\in\R^{2d}$. 

In order to describe the symplectic group, we first introduce the \textit{standard symplectic form} on $\mathbb{R}^{2d}$, which is the bilinear map $[\cdot, \cdot ] : \mathbb{R}^{2d} \times \mathbb{R}^{2d}$ given via
\[
[(x,\omega),(x',\omega')] = x' \cdot \omega - x \cdot \omega'~,
\] where $x,x',\omega,\omega' \in \mathbb{R}^d$. The standard symplectic matrix 
\begin{equation*}
    \J = \begin{pmatrix}
    0 & I_d \\
    -I_d & 0
    \end{pmatrix},
\end{equation*}
with $I_d$ denoting the $d$-dimensional identity matrix,
allows to express this bilinear form as $[\lambda,\lambda'] = \lambda \cdot \J \lambda'$. Now the symplectic group $\Sp\leq \mathrm{SL}(2d,\R)$ consists of the matrices $S\in \mathrm{GL}(2d,\R)$ preserving the commutator relation, i.e., they satisfy  
\begin{equation}\label{eq:S_defining}
         [\rho(S\lambda),\rho(S\nu)]_{-} = [\rho(\lambda),\rho(\nu)]_{-},\qquad \lambda, \nu \in\R^{2d},
\end{equation}
where $[\cdot,\cdot]_{-}$ denotes the commutator of linear operators.
One easily verifies that this condition is equivalent to the equation $S^T \J S = \J$. These characterizations readily imply that the symplectic group is indeed a closed subgroup of the full matrix group $\mathrm{GL}(2d,\R)$.

Due to the block structure of $\J$, symplectic matrices are often represented in block form as well. The following examples of symplectic matrices will turn out to be important building blocks of the symplectic group:
 let $P,\,Q\in\R^{d\times d}$ be symmetric and $L\in \mathrm{GL}(d,\R)$ be invertible. The following matrices are in $\Sp$:
\begin{equation}
U_P=\begin{pmatrix}
I & P \\
0 & I
\end{pmatrix},\quad 
V_Q=\begin{pmatrix}
I & 0 \\
Q & I
\end{pmatrix},\quad
D_L=\begin{pmatrix}
L & 0 \\
0 & L^{-T}
\end{pmatrix}.
\end{equation}
The first $d$ elementary symplectic quasi-permutation matrices $\Pi_{i}\in \Sp$,$\ 1\leq i\leq d$, are given by
\begin{equation}
\Pi_i\ e_j=\begin{cases}
e_j\ ,& \;1\leq j\leq 2d, j\neq i, j\neq i+d,\\
-e_{i+d}\ , &\; j=i, \\
e_i\ , &\; j=i+d.
\end{cases}
\end{equation}
We observe that each $\Pi_i$ acts via a combination of a permutation (in fact, a transposition) and a sign change. 
Notice that we can represent the standard symplectic matrix as
\begin{equation}
    \J = \prod\limits_{i=1}^{d}\Pi_i.
\end{equation}
For all $1\leq i\leq d$ holds
\begin{equation}
\Pi_{i+d}\coloneqq\Pi_{i}^T = \Pi_i^{-1}\;, \quad \Pi_{i+d}\ e_j =\begin{cases}
e_j\ ,& \;1\leq j\leq 2d, j\neq i, j\neq i+d,\\
e_{i+d}\ , &\; j=i, \\
-e_i\ , &\; j=i+d.
\end{cases} 
\end{equation}
Moreover, all $\Pi_i$ commute, i.e.  $\Pi_i \Pi_j = \Pi_j \Pi_i$ for all $1\leq i,j\leq 2d$.
These symplectic matrices are significant not only due to the clean definition but also for their generating properties.
 \begin{proposition}[Factorization in $\Sp$, \cite{DopicoJohnson2009}]\thlabel{prop:decomposition}
 For all $S\in \Sp $ there is an index set $J\subseteq \{1,...,2d\}$, symmetric matrices $P,Q\in\R^{d\times d}$ and an invertible matrix $L\in \mathrm{GL}(d,\R)$ with
\begin{equation}
S=\prod\limits_{i\in J}\Pi_i\; V_Q D_L U_P.
\end{equation}
 \end{proposition}

 The symplectic group can also be understood in terms of the Heisenberg group $\H$. More precisely, it can be equivalently defined as the group of linear automorphisms $\mathcal{A}$ on $\left(\R^{2d}, +\right)$ which extend to an automorphism $(\lambda,t)\mapsto (\mathcal{A}t,t)$ of the Heisenberg group \cite{Folland1989}.
 The property \eqref{eq:S_defining} and methods from representation theory, in particular Stone-von Neumann's theorem, imply the existence of unitary operators $\h{S}$
 with
\begin{equation}
    \rho(S\lambda) = \h{S}\rho(\lambda)\h{S}^{-1},
\end{equation}
that is, $\widehat{S}$ arises as an intertwining operator. They generate a group $G$ with the composition as the group multiplication. However, any multiple $\tau\widehat{S}$, $\tau\in\mathbb{T}$, is an intertwining operator, implying a high redundancy of the group action of $G$ on $L^2(\R^d)$. Nevertheless, there exists a subgroup of $G$ such that for all $S\in\Sp$ the said subgroup contains exactly two corresponding operators $\widehat{S}$.
The group is called the \textit{metaplectic group} and is denoted by $\Mp$. It is an explicit realization of the double cover of $\Sp$ and the projection
\begin{equation*}
    \pi^{\mathrm{Mp}}:\Mp\to\Sp
\end{equation*}
is a group homomorphism, with kernel $\mathrm{ker}(\pi^{\mathrm{Mp}}) =\left\lbrace \mathcal{\mathrm{id}, -\mathrm{id}}\right\rbrace$. The metaplectic operators are automorphisms on the Schwartz space, and can therefore be extended by duality to $\mathcal{S}'(\R^d)$ \cite[Proposition 4.27]{Folland1989}.
Standard examples of metaplectic operators are the unitary dilations $\mathcal{D}_L f(t) = \abs{\det{L}}^{-1} f(L^{-1}t)$, the linear chirps $\mathcal{V}_Q f(t) = e^{\pi i t\cdot Qt}f(t)$ and the Fourier transform. Indeed, 
\begin{equation}\begin{split}
    \rho(\J\lambda)&=\F\, \rho(\lambda)\,\F^{-1}, \\
    \rho(V_Q\lambda)&=\mathcal{V}_Q\,\rho(\lambda)\, \mathcal{V}_Q^{-1}, \\
    \rho(D_L\lambda)&=\mathcal{D}_L\,\rho(\lambda)\, \mathcal{D}_L^{-1}, \\
    \rho(U_P\lambda)&=\left(\F\, \mathcal{V}_{-P}\,\F^{-1}\right)\rho(\lambda) \left(\F\, \mathcal{V}_{-P}\,\F^{-1}\right)^{-1}.
\end{split}\end{equation}

 A slightly more explicit description is based on quadratic Fourier transforms. Detailed constructions can be found in \cite[Chapter 9]{Groechenig2001}, \cite[Chapter 4]{Folland1989} and \cite[Chapter 7]{Gosson2011}.
The connection between the ambiguity function and the metaplectic operators is best summarized by the symplectic covariance of the ambiguity function:
\begin{equation}\label{eq:symplectic_covariance}
\AF(f,g)\circ S^{-1}\ (\lambda) =  \AF\left(\h{S}f,\h{S} g \right)(\lambda).
\end{equation}
Rewriting this equation slightly as 
\[
A(\h{S}f,g) (\lambda) = A(f, \h{S}^{-1}(g)) \circ S^{-1}(\lambda)
\]
shows that the effect of $\h{S}$ on the ambiguity function amounts to exchanging the window (which induces a well-understood equivalence of modulation space norms, since the window class is invariant under the metaplectic action), and the composition operator obtained by letting $S^{-1}$ act on the input variables of the ambiguity function. Hence the focus of the following section will be on the latter. 

\section{The metaplectic action on \texorpdfstring{$L^{p,q}(\R^{2d})$}{Lebesgue spaces } and on \texorpdfstring{$\MSpace{p,q}{}$}{modulation spaces}}\label{sec:unweighted}
In this section we fully characterize the metaplectic operators leaving invariant the unweighted modulation spaces $\MSpace{p,q}{}$. This problem turns out to be rather closely related to a purely measure-theoretic question, namely invariance of the associated mixed unweighted spaces $L^{p,q}(\mathbb{R}^{2d})$ under symplectic changes of variable. Before we explain this aspect of our strategy in more detail, we introduce some notation concerning dilation operators and their domains and make some fundamental observations concerning the boundedness of such operators. 

\begin{proposition}\thlabel{prop:closed_operator} 
Let $S\in\R^{2d\times 2d}$ be an invertible matrix and $V\subseteq L^{p,q}(\R^{2d})$ a closed subspace. Assume
\begin{equation}
\mathcal{D}_S:V \to  L^{p,q}(\R^{2d}),\quad f\mapsto f\circ S^{-1}
\end{equation}
is well-defined. Then $\mathcal{D}_S$ is bounded.
\end{proposition}
\begin{proof}
$\mathcal{D}_S$ is obviously a linear operator. We will  show that it is also closed. To that end, let $(f_n, \mathcal{D}_S f_n)_{n\in\N}\subseteq V\times L^{p,q}(\R^{2d})$ be an arbitrary Cauchy sequence. Due to the completeness of $V\times L^{p,q}(\R^{2d})$, the sequence has a limit $(f,g)$. The convergence in the product space means
\begin{equation}
\norm{f_n-f}_{L^{p,q}}\to 0\qquad \text{ and } \qquad
\norm{\mathcal{D}_S f_n-g}_{L^{p,q}}\to 0.
\end{equation}
By the theorem of Riesz-Fischer, there is a subsequence $(f_{n_k})_{k\in\N}$ with $f_{n_k}(x)\to f(x)$ almost everywhere. Hence we get $\mathcal{D}_S f_{n_k}(x)\to f\circ S^{-1} (x)$ almost everywhere. On the other side, 
\begin{equation}
\norm{\mathcal{D}_S f_{n_k}-g}_{{p,q}}\to 0
\end{equation}
and we can choose a subsequence $(f_{n_{k_m}})_{m\in\N}$ with $\mathcal{D}_S f_{n_{k_m}}\to g\ $ almost everywhere. This implies $g=f\circ S^{-1}$ almost everywhere, which proves the claim. Since both $V$ and $L^{p,q}(\R^{2d})$ are Banach spaces, the closed graph theorem implies that $\mathcal{D}_S$ is bounded.
\end{proof} 
In the subsequent applications of this result, the closed subspace $V$ will be either all of $L^{p,q}(\mathbb{R}^{2d})$, or the image space $A(\cdot,g)(\MSpace{p,q}{})$ of the modulation space $\MSpace{p,q}{}$ under the ambiguity operator, with respect to a fixed nonzero window $g \in S(\R^d)$. 
Following standard terminology, we call $\mathcal{D}_S$ \textit{everywhere defined} on $V$ if 
\begin{equation}
    \dom (\mathcal{D}_S)\coloneqq\{f\in V \mid f\circ S^{-1}\in L^{p,q}(\R^{2d})\}= V.
\end{equation}
With this notation, the interplay between the symplectic action on $L^{p,q}(\R^{2d})$ and the metaplectic action on $\MSpace{p,q}{}$ can be best illustrated by the commutative diagram shown in Figure \ref{fig:comm_diag}. 

\begin{figure}[ht!] 
\[
  \begin{tikzcd}
    \MSpace{p,q}{}  \arrow{rr}{\AF(\dotp,g)} \arrow{dddd}[swap]{\widehat{S}}& & \AF(\dotp,g)(\MSpace{p,q}{} ) \arrow[rr,hookrightarrow]{}\arrow[dddd,dashed]{} && L^{p,q}(\R^{d}) \arrow{dddd}{\mathcal{D}_S} \\
    &&&&\\
    &&&&\\
    &&&&\\
    \MSpace{p,q}{}  \arrow{rr}[swap]{\AF(\dotp,\widehat{S}^{-1}g)}& & \AF(\dotp,\widehat{S}^{-1}g)(\MSpace{p,q}{} ) \arrow[rr,hookrightarrow]{} && L^{p,q}(\R^{d})  \\
  \end{tikzcd}
\]
\caption{Commutative diagram}\label{fig:comm_diag}
\end{figure}

We now state the main result of this section. 
\begin{theorem}\thlabel{thm:main_unweighted}
Let $p,q\in [1,\infty]$ and $\widehat{S}\in \Mp$ be given. The following statements are equivalent:
\begin{enumerate}[(a)]
\item $\widehat{S}:\MSpace{p,q}{} \to \MSpace{p,q}{} $ is everywhere defined.
\item $\widehat{S}:\MSpace{p,q}{} \to \MSpace{p,q}{} $ is everywhere defined and bounded.
\item One of the following conditions holds:
\begin{enumerate}[(i)]
\item $p=q$, or
\item $p\neq q$ and the projection $\pi^{\mathrm{Mp}}(\widehat{S})=S\in \Sp $ is an upper block triangular matrix.
\end{enumerate}
\end{enumerate}
\end{theorem}
This is a direct consequence of the following theorem.
\begin{theorem}\thlabel{thm:classification}
Let $p,q\in [1,\infty]$ and $S\in \Sp $ be given. The following statements are equivalent:
\begin{enumerate}[(a)]
\item $\mathcal{D}_S:L^{p,q}(\R^{d})\to L^{p,q}(\R^{d})$ is everywhere defined.
\item $\mathcal{D}_S:L^{p,q}(\R^{d})\to L^{p,q}(\R^{d})$ is everywhere defined and bounded.
\item $\mathcal{D}_S: \AF(\dotp,\widehat{S}^{-1}g)(\MSpace{p,q}{} )\to L^{p,q}(\R^{d})$ is everywhere defined.
\item $\mathcal{D}_S: \AF(\dotp,\widehat{S}^{-1}g)(\MSpace{p,q}{} )\to L^{p,q}(\R^{d})$ is everywhere defined and bounded.
\item One of the following conditions holds:
\begin{enumerate}[(i)]
\item $p=q$, or
\item $p\neq q$ and $S$ is an upper block triangular matrix.
\end{enumerate}
\end{enumerate}
\end{theorem}

Let us first comment on the implications that are obvious by now: The equivalences (a)$\Leftrightarrow$(b) and (c)$\Leftrightarrow$(d) are provided by Proposition \ref{prop:closed_operator}. The implication (b) $\Rightarrow$ (d) follows from the commutative diagram. Note that the converse (d) $\Rightarrow$ (b) appears somewhat unexpected at first. Given a densely defined closed operator $T: Y \to Y$ on a general Banach space, and a closed subspace $X \subsetneq Y$, there is no reason to expect that boundedness of the restriction $T|_X$ implies boundedness of $T$ itself, and counterexamples to this expectation are easily constructed. Nonetheless, in the setting studied here, the implication holds. 

Of the remaining implications required to complete the proof of the theorem, (e) $\Rightarrow$ (a) will turn out to be a fairly standard application of the linear change of variable formula; see the proof of \thref{thm:pq_bounded_unweighted} below. Hence the main challenge in the following will be the proof of the implication (d) $\Rightarrow$ (e). Here we will rely on the factorization of arbitrary symplectic matrices provided by \thref{prop:decomposition}.

But first, we deal with the case $p=q$. Since $L^{p,p}(\R^{2d}) = L^p(\R^{2d})$, this case is immediate by the standard change of variables formula for integrals. The only fact about symplectic matrices used here is the fact that their determinant equals $1$.

\begin{theorem}\thlabel{thm:p=q_unweighted}
The mapping $\mathcal{D}_S:L^p(\R^{2d})\to L^p(\R^{2d})$ is everywhere defined and norm-preserving for all $S\in \Sp $ and all $p\in [1,\infty]$.
\end{theorem}
As a consequence of this observation and the already established implication (a) $\Rightarrow$ (c), we get the boundedness of $\widehat{S}: \MSpace{p}{} \to \MSpace{p}{}$, for any symplectic matrix $S$. This was noted in \cite{CauliEtAl2019} but was probably known prior to that.

From now on, let $p,q\in [1,\infty]$ be distinct. In this case, condition (e) of \thref{thm:main_unweighted} contains  restrictions on the matrix $S$, which arise from the fact that the definition of the mixed $L^{p,q}$-norm treats the integration variables $x$ and $\omega$ in an asymmetric manner. Hence any change of variables that interacts with this asymmetry ("mixes" the $x$- and $\omega$-variables) can be expected to be problematic, i.e., potentially lead to unbounded operators. The challenge in the proof of the remaining direction lies in making this intuition precise. 

The following, somewhat extreme example, based on interchanging the variables $x,\omega$, gives a hint why the intuition is correct. Pick radially symmetric functions $f\in L^p(\R^{d})\setminus L^q(\R^{d})$ and $g\in L^q(\R^{d})\setminus L^p(\R^{d})$. The tensor product $(x,\omega)\mapsto f\otimes g(x,\omega)=f(x)g(\omega)$ is in $L^{p,q}(\R^{2d})$, but $g\otimes f = \w{\J}f\otimes g \notin L^{p,q}(\R^{2d})$, so intuitively, mixing of the variables should be held under control as much as possible. The example shows that $\w{\J}:L^{p,q}(\R^{2d})\to L^{p,q}(\R^{2d})$ is not well defined. 

The following result establishes the implication (e) $\Rightarrow$ (a) of \thref{thm:classification} for the case $p \not= q$.
\begin{theorem}\thlabel{thm:pq_bounded_unweighted}
Let $A,D\in \mathrm{GL}(d,\R)$ and $B\in\R^{d\times d}$ be given. Denote $S \coloneqq\begin{pmatrix}
A & B\\0 & D
\end{pmatrix}$, then 
\[ \mathcal{D}_S :L^{p,q}(\R^{2d})\to L^{p,q}(\R^{2d})\]
is, up to a constant $C_S$, a norm-preserving isomorphism with a bounded inverse $\mathcal{D}_S^{-1} = \mathcal{D}_{S^{-1}}$.
\end{theorem}
\begin{proof}
Let $f\in L^{p,q}(\R^{2d})$ be given.\newline
The inverse of $S$ is \[S^{-1} \coloneqq\begin{pmatrix}
A^{-1} & -A^{-1}BD^{-1}\\0 & D^{-1}
\end{pmatrix}.\]
Due to the translation invariance of the Lebesgue measure, it holds
\begin{equation}\begin{split}
\norm{\mathcal{D}_S f}_{{p,q}} {=}&\ \norm{\omega\mapsto\norm{f(A^{-1}\dotp -A^{-1}BD^{-1}\omega, D^{-1}\omega)}_p}_q \\
= &\ \norm{\omega\mapsto \abs{\det\ A}^{\frac{1}{p}}\norm{f(\dotp, D^{-1}\omega)}_p}_q \\
= &\ \abs{\det\ A}^{\frac{1}{p}}\ \abs{\det \ D}^{\frac{1}{q}}\norm{f}_{L^{p,q}},
\end{split}\end{equation}
where $a^{\pm\frac{1}{\infty}}\coloneqq1$. 

The inverse of $S$ is also an upper block triangular matrix with $C_{\SI} = C_S^{-1}$. It obviously holds $\mathcal{D}_S \mathcal{D}_{S^{-1}} f = \mathcal{D}_{S^{-1}} \mathcal{D}_S f = f$. Hence $\mathcal{D}_S$ is an isomorphism. 
\end{proof}

Before we finish the proof of our main theorem, we prove a lemma that reduces the case of general symplectic matrices $S$ to matrices having a specific structure. 
\begin{lemma} \thlabel{lem:special_symplectic}
Let $S \in \Sp$ be arbitrary, and $1 \le p,q \le \infty$. Then there exists symplectic matrices $S',\,S''$ with the following properties: 
\begin{enumerate}[(a)]
\item There exist $0 \le k \le d$ and matrices $Q',Q'' \in \R^{d \times d}$ such that 
\[
 S' =  \prod\limits_{i=k+1}^d \Pi_i\ V_{Q'},\qquad S''=  V_{Q''} \prod\limits_{i=k+1}^d \Pi_i  ~.
\]
\item $\mathcal{D}_S: \AF(\dotp,\widehat{S}^{-1}g)(\MSpace{p,q}{} )\to L^{p,q}(\R^{2d})$ is everywhere defined iff $\mathcal{D}_{S'}$ or $\mathcal{D}_{S''}$  has the same property.  
\item $S$ is upper block triangular iff $S'$ and $S''$ are.
\end{enumerate}
\end{lemma}

\begin{proof}
Recall the factorization 
\[
S = \prod\limits_{i\in K}\Pi_i\ V_Q D_L U_P~
\] 
from \thref{prop:decomposition}, and note that both $\mathcal{D}_{D_L}$ and $D_{U_P}$ are bounded operators with bounded inverses on $L^{p,q}(\mathbb{R}^{2d})$, by \thref{thm:p=q_unweighted}. Hence $\mathcal{D}_S: \AF(\dotp,\widehat{S}^{-1}g)(\MSpace{p,q}{} )\to L^{p,q}(\R^{2d})$ is everywhere defined iff $\mathcal{D}_{M}$ has the same property, where 
\[
M = \prod\limits_{i\in K}\Pi_i\ V_Q.
\] 
Furthermore, since the set of upper block triangular matrices is a subgroup containing $D_L$ and $U_P$, $M$ is upper block triangular iff $S$ is.

It remains to reduce the case of general index sets $K$ to $K= \{k+1,\ldots,d \}$. To begin with, $\Pi_i \Pi_{i+d} = I$, so we can assume that for no $i\in\{1,\cdots,d\}$ both $i$ and $i+d$ are in $I$. 
To separate the indices, we define 
\begin{equation}
    I_< =I\cap \{1,...,d\} \quad\text{and}\quad I_> = \{i\in \{1,\cdots,d\}\mid i+d\in I\}.
\end{equation}
By assumption, $I_<$ and $I_>$ are disjoint. We define the diagonal matrix $K_{I}\in\R^{d\times d}$, which serves to correct the sign:
\begin{equation}
    \left(K_{I}\right)_{j,j} = \begin{cases} 
-1&,\text{ if }\; j\in I_>,\\
1&,\text{ otherwise.}
\end{cases}
\end{equation}
This results in
\begin{equation}
    \prod\limits_{i\in I}\Pi_i = \prod\limits_{i\in I_>}\Pi_i^T \prod\limits_{i\in I_<}\Pi_i
= \diag(K_{I},K_I)\ \prod\limits_{i\in I_>\bigcup I_<}\Pi_i.
\end{equation}
Finally, we observe 
\begin{equation}
    \prod\limits_{i\in I}\Pi_i\ e_j \neq e_j\qquad\text{ if and only if}\qquad \prod\limits_{i\in I}\Pi_i\ e_{j+d}\neq e_{j+d}.
\end{equation}
So as not to have to deal with the order of the indices, we write 
\begin{equation}
    \prod\limits_{i\in I_>\bigcup I_<}\Pi_i = \diag(R_I,R_I)\prod\limits_{i=k+1}^d\Pi_i,
\end{equation}
where $k=d-\abs{I}$ and $R_I$ is the permutation matrix which represents the permutation of the indices in $I$ with $\{k+1,\cdots, d\}$. 
Altogether, we get the slightly more complicated decomposition
\begin{equation}\begin{split}
M&=\diag(K_{I},K_I)\ \diag(R_I,R_I)\prod\limits_{i=k+1}^d\Pi_i\ V_Q \\
&= \diag(K_{I}R_I,K_I R_I)\ \underbrace{\prod\limits_{i=k+1}^d\Pi_i\ V_Q}_{\eqqcolon S'}.
\end{split}\end{equation}
Hence, one more appeal to \thref{thm:pq_bounded_unweighted} provides that $\mathcal{D}_{S'}$ is everywhere defined iff the same holds for $\mathcal{D}_{S}$.
In addition, $S'$ is upper block triangular iff $S$ has the same property. Hence (a) and (b) are established for the factorization
\[
S' = \prod\limits_{i=k+1}^n\Pi_i\ V_Q~.
\]
For the verification of (c), observe that the equation 
\begin{equation}
\widetilde{S} = \diag(K_{I}R_I,K_I R_I) S'
\end{equation}
implies that $S'$ is upper block triangular iff $\widetilde{S}$ is, and we already established that this holds iff $S$ is upper block triangular. 

We have thus shown the desired statements for the first factorization in (a). The second factorization, with the associated properties, is obtained by applying the first factorization to the inverse matrix $\SI$, and taking inverses. Denoting by $I_l$ the indentity matrix in $\R^l$, we can express the inverse of $S'$ as
\begin{equation}
    (S')^{-1} = V_{-Q} \prod\limits_{i=k+1}^d\Pi_i^T = \underbrace{ V_{-Q} \prod\limits_{i=k+1}^d\Pi_i }_{\eqqcolon S''}\ \diag(I_k,-I_{d-k},I_k,-I_{d-k}).
\end{equation}
The matrix $S''$ satisfies the claims by analogous arguments.
\end{proof}

We can now finish the proof of our main theorem.
\begin{proof}[Proof of \thref{thm:classification}]
It remains to prove the implication (d) $\Rightarrow$ (e) of the theorem, for the case $p \not= q$. Assume that $S \in \Sp$ is not upper triangular. We intend to establish that $\mathcal{D}_S: \AF(\MSpace{p,q}{}, g) \to L^{p,q}(\R^{2d})$ is unbounded. Recall from the previous lemma that, without loss of generality, we can assume
\begin{equation} \label{eqn:S_special_form}
\SI = V_Q\prod\limits_{i=k+1}^d\Pi_i
\qquad \text{or}\qquad
\SI = \prod\limits_{i=k+1}^d \Pi_i\ V_Q .
\end{equation} 

The first factorization will be used to treat the case $p<q$, and the second one for the case $p>q$. If $Q=0$, this choice is irrelevant.

Now let $\varepsilon>1$, set $E\coloneqq\epsilon I\in\R^{d\times d}$ and denote with $\Delta\coloneqq(1-\varepsilon^{-2})^{\frac{1}{2}}I$. The matrices 
\begin{equation}
    E,\quad E^{-1},\quad(\varepsilon^2-1)^{\scriptscriptstyle \frac{1}{2}}I,\quad \Delta
\end{equation}

 are all symmetric positive definite, with $\Delta^2+E^{-2} = I$.

It holds
\begin{equation}
\AF\left(\ g\circ (\varepsilon^2-1)^{\frac{1}{2}}I, g\ \right)(x,\omega) = \abs{\det \ E}^{-1} e^{\pi i x\cdot \omega- 2\pi i \omega\cdot  E^{-2}x}e^{-\pi x\cdot \Delta^2 x} e^{- \pi\omega\cdot  E^{-2}\omega}. \label{eq:calculation_A_Gauss}
\end{equation}
The interested reader can find the detailed computation proving this equation in Appendix~\ref{app:calculation_A_Gauss}.

The term $\abs{\det \ E}^{-1} e^{\pi i x\cdot \omega- 2\pi i \omega\cdot  E^{-2}x}$ has no effect on the norm estimate. Therefore we focus on $f(x,\omega)\coloneqq e^{-\pi x\cdot \Delta^2 x} e^{- \pi\omega\cdot  E^{-2}\omega}$ and $f \circ \SI$. It holds
\begin{equation}\label{f-Norm}
\norm{f}_{{p,q}} = \norm{g\circ\Delta}_p \norm{g\circ E^{-1}}_q \myeq{\ref{lemma:NormFormel}}{\asymp} \det (\Delta^2)^{-\frac{1}{2p}}\det(E^{-2})^{-\frac{1}{2q}}.
\end{equation}
We denote the blocks of $\SI$ by
\begin{equation}
\SI 
=: \begin{pmatrix}
A & B\\C & D\end{pmatrix}
\end{equation}
and compute $f \circ \SI$ as 
\begin{equation}
(f \circ \SI)(x,\omega) = \exp \left(-\pi (Ax+B\omega)\cdot \Delta^2(Ax+B\omega)-\pi(Cx+D\omega)\cdot  E^{-2}(Cx+D\omega)\right).
\end{equation}
By expanding and grouping the quadratic terms, one can show that 
\begin{equation}
\norm{f \circ \SI }_{{p,q}} \asymp \abs{\det \ \Sigma}^{-\frac{1}{2p}} \abs{\det \ \Omega}^{-\frac{1}{2q}} \label{eq:calculation_unweighted_norm}
\end{equation}
for all $p,q\in [1,\infty],\ p\neq q$, where 
\begin{equation}\begin{split}
\Sigma&\coloneqq A^T\Delta^2A +C^T E^{-2}C, \\
\beta&\coloneqq B^T\Delta^2A+D^T E^{-2}C, \\
\Omega&\coloneqq B^T\Delta^2 B+D^T E^{-2}D-\beta\Sigma^{-1}\beta^T.
\end{split}\end{equation}
The technical calculation is provided in Appendix \ref{app:calculation_unweighted_norm}.

We now compute $\Sigma, \beta$, and $\Omega$ for a few extremal cases and compare the norms. Throughout, we denote with $I_k =\diag (1,\dots, 1,0,\dots, 0)\in\R^{d\times d}$ the diagonal matrix with $k$ ones and $d-k$ zeros. 
\\[.2cm]
\textbf{Case 1:} $\SI = \prod\limits_{i=k+1}^d\Pi_i\ V_0 = \begin{pmatrix}
I_k & I-I_k\\ I_k-I & I_k
\end{pmatrix}$.\\[.3ex]
We keep in mind that $\Delta,\,E\in\R I$, i.e., $\Delta$ and $ E $ are multiples of $I$, as well as the fact that $(I-I_k)$ and $I_k$ are diagonal matrices with $(I-I_k)\cdot I_k = 0$. It thereby holds
\begin{equation}\begin{split}
\Sigma&= A^T\Delta^2A +C^T E^{-2}C = I_k\Delta^2I_k+(I_k-I)E^{-2}(I_k-I) = \Delta^2 I_k +  E^{-2}(I-I_k), \\[.7ex]
\beta&= B^T\Delta^2A+D^T E^{-2}C = (I_k-I)\Delta^2 I_k +I_k E^{-2}(I-I_k) = 0+0=0 \\[.7ex]
&\text{and} \\[.7ex]
\Omega&= B^T\Delta^2 B+D^T E^{-2}D-\beta\Sigma^{-1}\beta^T\\
&= (I-I_k)\Delta^2 (I-I_k)+I_k E^{-2}I_k= \Delta^2 (I-I_k)+ E^{-2}I_k.\\[.7ex]
\end{split}\end{equation}
For the sake of contradiction, assume $\mathcal{D}_S$ were bounded. In that case, it would hold
\begin{equation}\begin{split}
1&\gtrsim \abs{\det \ \Delta^2}^{\frac{1}{2p}}\ \abs{\det\ E^{-2}}^{\frac{1}{2q}}\ 
\abs{\det \ \Sigma}^{-\frac{1}{2p}}\  \abs{\det \ \Omega}^{-\frac{1}{2q}} \\[.7ex]
&= (1-\varepsilon^{-2})^{\frac{d}{2p}}\ (\varepsilon^{-2})^{\frac{d}{2q}}\
(1-\varepsilon^{-2})^{-\frac{k}{2p}}\ (\varepsilon^{-2})^{-\frac{d-k}{2p}}\ 
(1-\varepsilon^{-2})^{-\frac{d-k}{2q}}\ (\varepsilon^{-2})^{-\frac{k}{2q}} \\[.7ex]
&= (1-\varepsilon^{-2})^{\frac{d}{2p}-\frac{k}{2p}-\frac{d-k}{2q}} (\varepsilon^{-2})^{\frac{d}{2q}-\frac{d-k}{2p}-\frac{k}{2q}} \\[.7ex]
&= (1-\varepsilon^{-2})^{(d-k)(\frac{1}{2p}-\frac{1}{2q})} \ (\varepsilon^{-2})^{(d-k)(\frac{1}{2q}-\frac{1}{2p})} \\[.7ex]
&=  \left(\frac{1-\varepsilon^{-2}}{\varepsilon^{-2}}\right)^{(d-k)(\frac{1}{2p}-\frac{1}{2q})} = \left(\varepsilon^2-1\right)^{(d-k)(\frac{1}{2p}-\frac{1}{2q})}.\\[.7ex]
\end{split}
\end{equation}
The matrix is not the identity matrix, so $k<d$. The exponent is not $0$ due to $p\neq q$. Depending on the sign of the exponent, by observing the expression as $\varepsilon\to\infty$ or $\varepsilon\searrow 1$, we conclude that the estimate cannot hold. Looking at $f$, it is clear that the exact index set of the permutations is irrelevant, only its cardinality matters.
\\[.5cm]
\textbf{Case 2:} $p<q$, $Q\neq 0$ and $\SI =V_{-Q}\ \prod\limits_{i=k+1}^d\Pi_i.$\\[.5ex]

Assume there were a $g_0\in \AF(\MSpace{p,q}{}, g)$ with
\begin{equation}
\abs{\w{V_Q}\, g_0} =\abs{\AF\left(g\circ(\varepsilon^2-1)^{\frac{1}{2}}I,\ g\right)}.
\end{equation}
In that case, holds
\begin{equation}
    \norm{g_0}_{{p,q}} =\norm{\abs{\det \  E}^{-1}f\circ V_Q}_{{p,q}} \quad\text{and}\quad \norm{ g_0 \circ \SI }_{{p,q}}=\norm{\abs{\det \  E}^{-1}f\circ \prod\limits_{i=k+1}^d\Pi_i}_{{p,q}}.
\end{equation}
According to Case 1, it holds

\begin{equation}
\norm{g_0 \circ \SI}_{{p,q}}
\asymp\abs{\det \  E}^{-1}(1-\varepsilon^{-2})^{-\frac{k}{2p}}\ (\varepsilon^{-2})^{-\frac{d-k}{2p}}\ 
(1-\varepsilon^{-2})^{-\frac{d-k}{2q}}\ (\varepsilon^{-2})^{-\frac{k}{2q}} .    
\end{equation}
For $\norm{g_0}_{{p,q}}$, we compute the corresponding $\Sigma,\ \beta$ and $\Omega$. To this end, we write
$Q=U\Lambda U^T$, where $U\in \mathrm{O}(d,\R)$ and $\Lambda$ diagonal, and use $\Delta,\,E\in\R I$ to obtain
\begin{align}
\Sigma&= A^T\Delta^2A +C^T E^{-2}C = \Delta^2+ Q E^{-2}Q = U\left(I- E^{-2} +\Lambda^2 E^{-2}\right) U^T\  \\[.7ex]
&=U\left(\diag\left(1-\varepsilon^{-2}+\Lambda^2_{i,i}\ \varepsilon^{-2}\right)_{1\leq i\leq d}\right) U^T,\\[.5ex]
&\text{ hence} \\
\Sigma^{-1}&= U\ \diag \left(\frac{1}{1-\varepsilon^{-2}+\Lambda^2_{i,i}\ \varepsilon^{-2}}\right)_{1\leq i\leq d}\ U^T, \\[.7ex]
\beta&= B^T\Delta^2A+D^T E^{-2}C =  E^{-2} Q ~, \\[.7ex]
\Omega&=\ B^T\Delta^2 B+D^T E^{-2}D-\beta\Sigma^{-1}\beta^T=  E^{-2} -  E^{-2} Q \Sigma^{-1} Q E^{-2}= \\[.7ex]
&=\  E^{-2}\left(I-U\Lambda U^T\ U\ \diag \left(\frac{1}{1-\varepsilon^{-2}+\Lambda^2_{i,i}\ \varepsilon^{-2}}\right)_{1\leq i\leq d}\ U^T\ U\Lambda U^T\  E^{-2}\right) \\[.7ex]
&=\  E^{-2}\ U\left(I-\ \diag \left(\frac{\Lambda^2_{i,i}\varepsilon^{-2}}{1-\varepsilon^{-2}+\Lambda^2_{i,i}\ \varepsilon^{-2}}\right)_{1\leq i\leq d}\right)\ U^T \\[.7ex]
&=\ U\  E^{-2}\ \diag \left(\frac{1-\varepsilon^{-2}+\Lambda^2_{i,i}\ \varepsilon^{-2}\ -\Lambda^2_{i,i}\varepsilon^{-2}}{1-\varepsilon^{-2}+\Lambda^2_{i,i}\ \varepsilon^{-2}}\right)_{1\leq i\leq d}\ U^T \\[.7ex]
&=\ U\  E^{-2}\Delta^{2}\ \diag \left(\frac{1}{1-\varepsilon^{-2}+\Lambda^2_{i,i}\ \varepsilon^{-2}}\right)_{1\leq i\leq d}\ U^T.
\end{align}
Assume $\mathcal{D}_S$ were bounded. Then it would hold
\begin{align}
1&\gtrsim \norm{g_0}^{-1}_{{p,q}}\norm{g_0 \circ \SI}_{{p,q}}\asymp
\abs{\det \  E}\abs{\det \ \Sigma}^{\frac{1}{2p}}\  \abs{\det  \ \Omega}^{\frac{1}{2q}}\;\norm{g_0 \circ \SI}_{{p,q}}\\[.6ex]
&= \left(\prod\limits_{i=1}^d \Lambda^2_{i,i}\varepsilon^{-2}+(1-\varepsilon^{-2})\right)^{\frac{1}{2p}}\ \left(\prod\limits_{i=1}^d 
\frac{1}{\Lambda^2_{i,i}\varepsilon^{-2}+(1-\varepsilon^{-2})}
\right)^{\frac{1}{2q}} 
 \\[.6ex]
&\qquad\qquad\qquad\qquad\qquad\qquad\qquad\qquad(1-\varepsilon^{-2})^{\frac{d}{2q}}(\varepsilon^{-2})^{\frac{d}{2q}} \abs{\det \  E}\cdot \norm{\ST g_0}_{{p,q}}
\\[.3ex]
&=\left(\prod\limits_{i=1}^d \Lambda^2_{i,i}\varepsilon^{-2}+(1-\varepsilon^{-2})\right)^{\frac{1}{2p}-\frac{1}{2q}}\ (1-\varepsilon^{-2})^{\frac{d}{2q}}(\varepsilon^{-2})^{\frac{d}{2q}} \\
&\qquad\qquad\qquad\qquad(1-\varepsilon^{-2})^{-\frac{k}{2p}}\ (\varepsilon^{-2})^{-\frac{d-k}{2p}}\ 
(1-\varepsilon^{-2})^{-\frac{d-k}{2q}}\ (\varepsilon^{-2})^{-\frac{k}{2q}}\qquad\qquad \\[.6ex]
&=\left(\prod\limits_{i=1}^d \Lambda^2_{i,i}\varepsilon^{-2}+(1-\varepsilon^{-2})\right)^{\frac{1}{2p}-\frac{1}{2q}}\ (1-\varepsilon^{-2})^{-k(\frac{1}{2p}-\frac{1}{2q})}(\varepsilon^{-2})^{-(d-k)(\frac{1}{2p}-\frac{1}{2q})} \\
&=\left(\ \prod\limits_{i=1}^k \frac{\Lambda^2_{i,i}\varepsilon^{-2}+(1-\varepsilon^{-2})}{1-\varepsilon^{-2}}\right)^{\frac{1}{2p}-\frac{1}{2q}}\cdot\quad \left(\ \prod\limits_{i=k+1}^d \frac{\Lambda^2_{i,i}\varepsilon^{-2}+(1-\varepsilon^{-2})}{\varepsilon^{-2}}\right)^{\frac{1}{2p}-\frac{1}{2q}} \\
&=\left(\prod\limits_{i=1}^k \frac{\Lambda^2_{i,i}}{\varepsilon^2-1}+1\right)^{\frac{1}{2p}-\frac{1}{2q}}\; \left(\prod\limits_{i=k+1}^d \Lambda^2_{i,i}+\varepsilon^2-1\right)^{\frac{1}{2p}-\frac{1}{2q}}.
\end{align}
\newline
The exponent $\frac{1}{2p}-\frac{1}{2q}$ is positive. If $k<d$, then the second product is not empty. The first product is bounded from below by $1$. Since
\begin{equation}
\lim\limits_{\varepsilon\to\infty} \prod\limits_{i=k+1}^d \Lambda^2_{i,i}+\varepsilon^2-1=\infty,
\end{equation}
we meet a contradiction.

If $k=d$, then the second product is empty, i.e., it equals $1$. The symmetric matrix $Q$ has at least one eigenvalue $\Lambda_{i,i}\neq 0$. In this case,
\begin{equation}
\lim\limits_{\varepsilon\to 1} \prod\limits_{j=1}^d \frac{\Lambda^2_{j,j}}{\varepsilon^2-1}+1\geq 1^{d-1} \lim\limits_{\varepsilon\to 1} \frac{\Lambda^2_{i,i}}{\varepsilon^2-1}+1 = \infty,
\end{equation}
contradicts the assumption.
\newline
Overall, $\mathcal{D}_S$ cannot be bounded.
\\[.5cm]
\textbf{Case 3:} $p>q$, $Q\neq 0$ and $\SI =\prod\limits_{i=k+1}^d\Pi_i^T\ V_{Q}.$\\[.5ex]
Similarly to the previous case, we are looking for $h_0\in \AF(\dotp, g)(\Omega^{p,q}(\R^{d}))$ with
\begin{equation}
    \abs{h_0 \circ \prod\limits_{i=k+1}^d\Pi_i}=\abs{ \AF\left(\ g\circ(\varepsilon^2-1)^{\frac{1}{2}}I,\ g\ \right)}.
\end{equation}
Then, it holds
\begin{equation}
    \norm{h_0}_{{p,q}} =\norm{\abs{\det \  E}^{-1}f\circ \prod\limits_{i=k+1}^d\Pi_i}_{{p,q}} \quad\text{and}\quad \norm{h_0 \circ \SI}_{{p,q}}=\norm{\abs{\det \  E}^{-1}f\circ V_Q}_{{p,q}}.
\end{equation}
Under the assumption that $\mathcal{D}_S$ is bounded, the computation in the second step would imply 
\begin{equation}\begin{split}
1&\gtrsim \norm{h_0}^{-1}_{{p,q}}\norm{\ST\, h_0}_{{p,q}} \asymp\left(\prod\limits_{i=1}^k \frac{\Lambda^2_{i,i}}{\varepsilon^2-1}+1\right)^{\frac{1}{2q}-\frac{1}{2p}}\; \left(\prod\limits_{i=k+1}^d \Lambda^2_{i,i}+\varepsilon^2-1\right)^{\frac{1}{2q}-\frac{1}{2p}}.
\end{split}\end{equation}
We have already seen that this cannot be true. Therefore, $\mathcal{D}_S$ is unbounded.
\\[.5cm]
\textbf{Step 4: (The existence of $g_0$ and $h_0$)} 
\newline
It is known that $\varphi\in \mathcal{S}(\R^{d})$ is equivalent to $\AF(\varphi,\gamma )\in \mathcal{S}(\R^{2d})$ for all $\gamma\in\mathcal{S}(\R^{d})$, in particular, $\AF(\varphi,\gamma)\in L^{p,q}(\R^{2d})$.
\newline
For a given $R\in \Sp $ and $\varphi\in \mathcal{S}(\R^{d})$, we are looking for a $\psi\in \MSpace{p,q}{}$ with
\begin{equation}
     \abs{\AF(\varphi,g)}=\abs{\w{R}\AF(\psi,g)}=\abs{\AF(\widehat{R}{\psi},\widehat{R}g)}.
\end{equation}
We can actually solve
\begin{equation}
    \AF(\widehat{R}{\psi},\widehat{R}g) = \AF(\varphi,g)
\end{equation} for $\psi$. For that, we choose $\gamma\in \mathcal{S}(\R^{d})$ with $\left\langle \gamma, g\right\rangle\neq 0\neq \left\langle \gamma, \widehat{R} g\right\rangle$. The function $\gamma$ exists because $\mathcal{S}(\R^{d})$ is not two-dimensional. With the inversion formula, we get 
\begin{equation}
    \AF(\dotp,\gamma)^*\AF(\dotp,\widehat{R}g) \widehat{R}\psi = 
\AF(\dotp,\gamma)^*\AF(\dotp, g) \varphi,
\end{equation}
i.\,e.,
\begin{equation}
    \left\langle \gamma,\widehat{R}g\right\rangle^{-1} \widehat{R}\psi = \left\langle \gamma,g\right\rangle^{-1} \varphi.
\end{equation}
In our case, $\varphi = g\circ (\varepsilon^2-1)^{\scriptscriptstyle \frac{1}{2}}I$, so the claim can be applied to $g_0$ and $h_0$. A possible choice would be 
\begin{equation}
    \psi = \frac{ \left\langle \gamma,\widehat{R}g\right\rangle}{\left\langle \gamma,g\right\rangle}\widehat{R}^{-1}\varphi.
\end{equation}
To summarize,
\begin{equation}
    \mathcal{D}_S:\AF(\dotp,\gamma)(\MSpace{p,q}{})\to L^{p,q}(\R^{2d})
\end{equation}
cannot be bounded, hence it cannot be everywhere defined on  $\AF(\dotp,\gamma)(\MSpace{p,q}{})$.
\end{proof}
\section{Weighted Spaces}\label{sec:weighted}
\subsection{The Lifting on \texorpdfstring{$L^{p,q}_m(\R^{2d})$}{weighted Lebesgue spaces}}
Figure \ref{fig:comm_diag_wght} contains the weighted version of the fundamental commutative diagram.
\begin{figure}[ht] 
\[
  \begin{tikzcd}
    \MSpace{p,q}{m}  \arrow{rr}{\AF(\dotp, g)} \arrow{dddd}[swap]{\widehat{S}}& & \AF(\dotp,g)(\MSpace{p,q}{m} ) \arrow[rr,hookrightarrow]{}\arrow[dddd,dashed]{} && L^{p,q}_m(\R^{2d}) \arrow{dddd}{(\mathcal{D}_S)_m} \\
    &&&&\\
    &&&&\\
    &&&&\\
    \MSpace{p,q}{m}  \arrow{rr}[swap]{\AF(\dotp,\widehat{S}^{-1}g)}& & \AF(\dotp, \widehat{S}^{-1}g)(\MSpace{p,q}{m} ) \arrow[rr,hookrightarrow]{} && L^{p,q}_m(\R^{2d})  \\
  \end{tikzcd}
\]
\caption{Weighted version of the commutative diagram} \label{fig:comm_diag_wght}
\end{figure}

The spaces $L^{p,q}(\R^{2d})$ and $L^{p,q}_m(\R^{2d})$ are obviously isomorphic via
\begin{equation}
    \Phi_m:L^{p,q}_m(\R^{2d})\to L^{p,q}(\R^{2d}),\quad f\mapsto f\cdot m.
\end{equation}
In contrast to the pair $(L^{p,q}(\R^{2d}), L^{p,q}_m(\R^{2d}))$, the relationship $(\MSpace{p,q}{} , \MSpace{p,q}{m} )$ is not that obvious, requiring the formulation and proof of fairly intricate \textit{lifting theorems} (cf. \cite{DoerflerGroechenig2011, GroechenigToft2011, GroechenigToft2013} for isomorphisms between weighted modulation spaces).  Abstract coorbit theory does imply that the modulation spaces $\MSpace{p,q}{m_1}$, $\MSpace{p,q}{m_2}$ are isomorphic for two moderate weights $m_1,\, m_2$, but this observation provides little help when it comes to the particular test we have on the block structure of the projection of $\widehat{S}$. 

Before we turn to these rather subtle questions, we first focus on the purely measure-theoretic case, i.e., question whether the dilation operator acting on the weighted space 
\begin{equation}
    (\mathcal{D}_S)_m:L^{p,q}_m(\R^{2d})\to L^{p,q}_m(\R^{2d}),\qquad f\mapsto f\circ \SI
\end{equation} is bounded. Using the isomorphism $\Phi_m$ from above, we compute 
\begin{equation}\label{eq:weight_transformation}
\Phi_m\ (\mathcal{D}_S)_m\ \Phi_m^{-1} f = m\cdot\left(\left(\tfrac{f}{m}\right)\circ \SI\right) =m\cdot\ \frac{f\circ \SI}{m\circ \SI} = \frac{m}{m\circ \SI}\cdot \mathcal{D}_S f,
\end{equation}
where, as in the previous section, $\mathcal{D}_S$ denotes the dilation operator acting on $L^{p,q}(\R^{2d})$. 
\begin{theorem}\thlabel{thm:multiplication_operator_M_m}
Let $m$ be a moderate weight function. Let $S\in \Sp $ be given. The following statements hold:
\begin{enumerate}[(a)]
\item If $\mathcal{D}_S: L^{p,q}(\R^{2d}) \to L^{p,q}(\R^{2d})$ is bounded, then $(\mathcal{D}_S)_m$ is bounded if and only if 
\begin{equation}
R_m\coloneqq\esssup\limits_{z\in\R^{2d}}\  {\frac{m(z)}{m(\SI z)}}<\infty.
\end{equation}
\item If $\mathcal{D}_S: L^{p,q}(\R^{2d}) \to L^{p,q}(\R^{2d})$ is unbounded and 
\begin{equation}
    T_m\coloneqq\essinf\limits_{z\in\R^{2d}}\ \frac{m(z)}{m (\SI  z)}>0,
\end{equation}
then $(\mathcal{D}_S)_m$ is unbounded.
\end{enumerate}
\end{theorem}
\begin{proof}
\begin{enumerate}[(a)]
\item If $\mathcal{D}_S$ is bounded, it is in fact an automorphism, by Theorem \ref{thm:classification}. Then Equation \eqref{eq:weight_transformation} implies that $ (\mathcal{D}_S)_m $ is bounded if and only if the multiplication operator
\begin{equation}
    M_m: L^{p,q}\to L^{p,q},\qquad f\mapsto \frac{m}{m\circ \SI}\dotp f
\end{equation}
is bounded. A straightforward, somewhat tedious computation shows that $M_m$ is bounded if and only if $R_m$ is finite.

\item According to \thref{prop:closed_operator}, there is a $f\in L^{p,q}(\R^{2d})$ with $\ST f\notin L^{p,q}(\R^{2d})$. It holds
\begin{equation}
\norm{M_m\ \ST f}_{p,q} \geq T_m \norm{\ST f}_{p,q}=\infty.
\end{equation}
\end{enumerate}
\end{proof}
\begin{remark}
If $\ST$ is unbounded, but $T_m=0$, then it is in general difficult to make a reasonable claim. One can find bounds $\eta_r>0$, such that
\[\frac{1}{\eta_r}<\frac{m}{m\circ \SI}<\eta_r \] 
on $[-r,r]^{2n}$ (cf. \cite[Lemma 11.1.1.]{Groechenig2001}). The problem lies on the behaviour of $\eta_r$ as $r$ tends to $\infty$. If it is particularly slow, then we could expect $M_m\,\ST$ to be unbounded. It can, though, also happen that $\frac{m}{m\circ S}$ decays so fast, that it compensates for the growth of $\ST f$ and 
\begin{equation}
    M_m\, \ST: L^{p,q}\to L^{p,q}
\end{equation}
is well-defined. According to \thref{prop:closed_operator}, this would suffice for the boundedness of $(\ST)_m$.
\end{remark}
Observe that the case $m\asymp m\circ S^{-1}$ is equivalent to $0< T_m \le R_m < \infty$. In this case, the positive results from \thref{thm:multiplication_operator_M_m} yield the following corollary.
\begin{corollary}\thlabel{cor:equiv_weight}
Let $m$ be a moderate weight and $S\in\Sp$ a symplectic matrix. Assuming $m\asymp m\circ \SI$, the operator
\begin{equation}
    \mathcal{D}_S: L^{p,q}\to L^{p,q}
\end{equation} is bounded if and only if the operator
\begin{equation}
    (\mathcal{D}_S)_m: L^{p,q}_m\to L^{p,q}_m
\end{equation}
is bounded.
\end{corollary}

In light of this, we prove an analogous version for weighted modulation spaces.
\subsection{The Lifting on \texorpdfstring{$\MSpace{p,q}{m}$}{weighted modulation spaces}} To establish the connection between the metaplectic action on unweighted and on weighted modulation spaces, we make use of \emph{Toeplitz operators}. 
\begin{definition}
Let $a\in \mathcal{S}(\R^{2d})$ be a symbol and $g\in \mathcal{S}(\R^{2d})$ a fixed window. Then the Toeplitz operator $\Tp_g (a)$ is defined by the formula
\begin{equation}
\left\langle \Tp_g(a)f_1,f_2\right\rangle_{\tiny L^2(\R^{d})} = \left\langle a \mathcal{V}_g f_1, \mathcal{V}_{g} f_2\right\rangle_{\tiny L^2(\R^{2d})} = \left\langle a\, \AF(f_1,g), \AF(f_2,g) \right\rangle_{\tiny L^2(\R^{2d})}
\end{equation}
for all $f_1,f_2\in L^2(\R^{d})$. It is a well-defined operator which extends uniquely to a continuous operator from $\mathcal{S}'(\R^{d})$ to $\mathcal{S}(\R^{d})$. The class of admissible symbols can be significantly extended (cf. \cite[Proposition 1.5.]{GroechenigToft2011}). In particular, we can choose the symbol to be our weight function.
\end{definition}
\begin{theorem}[Gröchenig, Toft \cite{GroechenigToft2011}]\label{thm:lifting}
Assume that $m$ is an even, $v$-moderate weight function, for the weight $v(x) = (1+\|x\|)^N$, and $g\in\mathcal{S}(\R^{d})$. Then the Toeplitz operator $\Tp_g(m)$
is an isomorphism from $\MSpace{p,q}{m_0}$ onto $\MSpace{p,q}{m_0/m}$ for every $v$-moderate even weight $m_0$ and every $p,q\in[1,\infty]$.
\end{theorem}
We will apply this theorem with $m_0=m$. Before that, we rewrite the defining property of $\Tp_g(a) f$ in terms of its ambiguity function. We recall that the standard symplectic form $[\gamma,\lambda]=\gamma\cdot  \J\lambda$ on $\R^{2d}$ and the symmetric time-frequency shifts satisfy 
\begin{equation}
    \rho(\lambda)\rho(\gamma) = \rho(\lambda+\gamma)e^{\pi i[\gamma, \lambda]}.
\end{equation}
From the definition of the ambiguity function, one easily sees that $\overline{\AF(g,g)(\lambda)}= \AF(g,g)(-\lambda)$ holds for all $\lambda\in \R^{2d}$. Let now $g\in\mathcal{S}(\R^d),\,g\neq 0$ be an arbitrary window and $a:\R^{2d}\to \C$ a symbol. Then it holds

\begin{equation}
    \begin{split}
     A(\Tp_g(a) f, g)(\lambda) & = \left\langle \Tp_g(a) f, \rho(\lambda) g\right\rangle \\
    & = \left\langle a \AF(f,g), \AF(\rho(\lambda) g, g)\right\rangle \\
    & = \int_{\R^{2d}} a(\gamma) \AF(f,g)(\gamma) \overline{\left\langle \rho(\lambda) g, \rho(\gamma) g\right\rangle }\,d\gamma \\
    & = \int_{\R^{2d}} a(\gamma) \AF(f,g)(\gamma) \overline{ \left\langle  g,\rho(-\lambda) \rho(\gamma) g\right\rangle} \,d\gamma \\
    & = \int_{\R^{2d}} a(\gamma) \AF(f,g)(\gamma) \overline{ \left\langle  g,\rho(-\lambda + \gamma) g\right\rangle e^{-\pi i[\gamma,\lambda]}} \,d\gamma \\
    & = \int_{\R^{2d}} a(\gamma) \AF(f,g)(\gamma) e^{\pi i[\gamma,\lambda]} \AF(g,g)(\lambda-\gamma) e^{\pi i[\gamma,\lambda]} \,d\gamma.   
    \end{split}
\end{equation} 
This shows that the Toeplitz operator can be understood through the \emph{twisted convolution}. The twisted convolution $F\natural G$ of two measurable functions $F, \,G:\R^{2d}\to \C$ is given by
\begin{equation}
F\natural G(\lambda)\coloneqq \int_{\R^{2d}}F(\gamma) G(\lambda-\gamma)e^{\pi i [\gamma,\lambda]}\ d\gamma,
\end{equation}
assuming the integral is well-defined. In terms of the Toeplitz operator $\Tp_g(a)$, we can now express its action as
\begin{equation}\label{eq:TpTwisted}
    \AF(\Tp_g(a) f,g) = \left( a\,\AF(f,g)\right)\natural \AF(g,g).
\end{equation}

\begin{corollary}
Let $m$ be a moderate weight, $p,q\in[1,\infty]$. If $\widehat{S}:\MSpace{p,q}{} \to \MSpace{p,q}{} $ is bounded and 
\[C_m\coloneqq \esssup\limits_{z\in\R^{2d}}{\frac{m(z)}{m(\SI z)}}<\infty,\]
then 
\[ \widehat{S}:\MSpace{p,q}{m} \to \MSpace{p,q}{m} \]
is well-defined and bounded.
\end{corollary}
\begin{proof}
Theorems \ref{thm:classification} and \ref{thm:multiplication_operator_M_m} imply that $(\mathcal{D}_S)_m:L^{p,q}_m\to L^{p,q}_m$ is bounded. The commutative diagram completes the proof.
\end{proof}
We now focus on the opposite implication.
\begin{theorem}\thlabel{thm:main_weighted}
Let $m$ be a moderate weight, $p,q\in[1,\infty]$. If $\widehat{S}\in \Mp$ with projection $\pi^{\tiny Mp}(\widehat{S})=S$ satisfies $m\asymp m\circ S^{-1}$, then $\widehat{S}$ is a bounded operator from $\MSpace{p,q}{} $ to $\MSpace{p,q}{} $ if and only if it is a bounded operator from $\MSpace{p,q}{m} $ to $\MSpace{p,q}{m} $.
\end{theorem}
\begin{proof}
The first implication follows from the last corollary. Let now $\widehat{S}$ be a bounded operator from $\MSpace{p,q}{m} $ to $\MSpace{p,q}{m} $. We fix as a window the normalized Gaussian $g\in\mathcal{S}(\R^{d})$ and recall the symplectic covariance of the ambiguity function \eqref{eq:symplectic_covariance}.
For all $f\in \MSpace{p,q}{} $ holds
\begin{equation} 
    \begin{split}
        \norm{\widehat{S} f}_{\MSpace{p,q}{} } &= \norm{\AF(\widehat{S}f,g)}_{L^{p,q}} \myeq{\tiny\text{Lift}}{\underset{\tiny \eqref{eq:TpTwisted}}{\asymp}} \norm{\left(\tfrac{1}{m}\AF(\widehat{S}f,g)\right)\ \natural\ \AF(g,g)}_{L^{p,q}_m} \\[1ex]
\overset{\eqref{eq:symplectic_covariance}}{\textcolor{white}{=}}&= \norm{\left(\tfrac{1}{m} \mathcal{D}_S  \AF(f,\widehat{S}^{-1}g)\right)\ \natural\ \AF(g,g)}_{L^{p,q}_m} \\[1ex]
&= \norm{\left(
\tfrac{m\circ S^{-1}}{ m\circ S^{-1}}\tfrac{1}{\left( m \circ S\circ \right)S^{-1}} \ST\  \AF(f, \widehat{S}^{-1}g)\right)\ \natural\ \AF(g,g)}_{L^{p,q}_m} \\[1ex]
&= \norm{ \left((\ST)_{m\widetilde{m}^{-1}}\ (\widetilde{m}m^{-1} \AF(f,\widehat{S}^{-1}g)) \right)\ \natural\ \AF(g,g)}_{L^{p,q}_m}, \\
&= \norm{(\ST)_{m\widetilde{m}^{-1}} 
(\ST)_{m\widetilde{m}^{-1}}^{-1}\; \left(\left((\ST)_{m\widetilde{m}^{-1}}\ (\widetilde{m}m^{-1} \AF(f,\widehat{S}^{-1}g)) \right)\ \natural\ \AF(g,g)\right)}_{L^{p,q}_m},
    \end{split}
\end{equation}
where $\widetilde{m} = \tfrac{m}{\ST^{-1}m}$.
Recall that the symplectic matrices are exactly those which preserve the standard symplectic form. Therefore, the twisted convolution of two functions satisfies 
\begin{equation} 
    \begin{split}
        \left(F\natural G\right) (S^{-1} \lambda) &= \int_{\R^{2d}} F(\gamma)G(S^{-1}\lambda-\gamma)e^{\pi i \sigma(\gamma,S^{-1}\lambda)}\ d\gamma\\[1ex]
&= \int_{\R^{2d}} F(S^{-1}S\gamma)G(S^{-1}(\lambda-S\gamma))e^{\pi i \sigma(S\gamma,SS^{-1}\lambda)}\ d\gamma\\[1ex]
=& \int_{\R^{2d}} F(S^{-1}\gamma)G(S^{-1}(\lambda-\gamma))e^{\pi i \sigma(\gamma,\lambda)}
\ d\gamma= (F \circ S^{-1})\,\natural\,  (G\circ S^{-1}) (\lambda).
    \end{split}
\end{equation}
This, along with the symplectic covariance of the ambiguity function \eqref{eq:symplectic_covariance}, leads to
\begin{equation}
    \begin{split}
        \norm{\widehat{S} f}_{\MSpace{p,q}{} } \asymp &\norm{(\ST)_{m\widetilde{m}^{-1}}\;\left( \left(\widetilde{m}m^{-1}\cdot \AF(f,\widehat{S}^{-1}g) \right)\ \natural\ \AF(\widehat{S}^{-1}g,\widehat{S}^{-1}g)\right)}_{L^{p,q}_m}.
    \end{split}
\end{equation}
We recall that $\widehat{S}^{-1}g$ is a Schwartz function. By the lifting Theorem \ref{thm:lifting} and Equation \eqref{eq:TpTwisted}, there exists a tempered distribution $\widetilde{f}\in \MSpace{p,q}{\widetilde{m}^{-1}m}$ with
\begin{equation}
    \left(\widetilde{m}m^{-1}\cdot \AF(f,\widehat{S}^{-1}g) \right)\ \natural\ \AF(\widehat{S}^{-1}g,\widehat{S}^{-1}g) = \AF(\widetilde{f},\widehat{S}^{-1}g).
\end{equation}
Recall that Theorem \ref{thm:lifting} also provides the norm equivalence $\| \tilde{f} \|_{M^{p,q}_{\widetilde{m}^{-1}m}}  \asymp \| f \|_{M^{p,q}}$, where the implied constant is independent of $f \in \MSpace{p,q}{}$.
By assumption, $\widetilde{m}\asymp 1$, which implies $\MSpace{p,q}{\widetilde{m}^{-1}m} = \MSpace{p,q}{m}$ and
\begin{equation}
\norm{\widetilde{f}}_{\MSpace{p,q}{m} }\asymp \norm{\widetilde{f}}_{\MSpace{p,q}{\widetilde{m}^{-1}m}}\asymp \norm{f}_{\MSpace{p,q}{} }.
\end{equation}
Putting it all together, we obtain
\begin{equation}
    \begin{split}
        \norm{\widehat{S} f}_{\MSpace{p,q}{} } \asymp &\norm{(\ST)_m \AF(\widetilde{f},\widehat{S}^{-1}g)}_{L^{p,q}_m} = \norm{\AF(\widehat{S}\widetilde{f},g)}_{L^{p,q}_m}\\[1ex] 
\asymp& \norm{\widehat{S}\widetilde{f}}_{\MSpace{p,q}{m} }  \lesssim \norm{\widetilde{f}}_{\MSpace{p,q}{m} } \asymp \norm{f}_{\MSpace{p,q}{} }.
    \end{split}
\end{equation}
This proves the claim.
\end{proof}

\subsection{Examples}
We consider two of the standard weights which are frequently used in time-frequency analysis and determine the cases where our criteria are applicable. We refer to \cite{Groechenig2007} for a general overview of typical weights.
\begin{example}
We begin with
\begin{equation}
m(z)= m_{s,t}(z) = (1+\norm{z})^s(\ln(e+\norm{z}))^t.
\end{equation}
This class contains in particular the radial polynomial weights.
Here it is well-known that $\| z \| \asymp \| S z \|$ holds for \textit{any} invertible matrix $S \in \R^{2d \times 2d}$. This immediately shows the equivalence $m_{s,0} \asymp m_{s,0} \circ S$.

For the logarithmic part, we spell out the norm equivalence as 
\[
 A \| z \| \le \| S z \| \le B \| z \|~, 
\] with $0< A\leq B$, and obtain 
\begin{equation}
    \begin{split}
        \frac{\ln(e+\norm{z})}{\ln(e+B\norm{z})} \leq \frac{\ln(e+\norm{z})}{\ln(e+\norm{Sz})}\leq  \frac{\ln(e+\norm{z})}{\ln(e+A\norm{z})},
    \end{split}
\end{equation}
We next establish bounds for the ends of this chain of inequalities as $\|z \| \to 0$ and $\| z \| \to \infty$. Using l'Hospital's rule allows computing the limits 
\begin{equation}
\lim\limits_{z\to 0} \frac{\ln(e+\norm{z})}{\ln(e+A \norm{z})}= \lim\limits_{z\to 0} \frac{\ln(e+\norm{z})}{\ln(e+B \norm{z})}=  1~,
\end{equation}
as well as 
\begin{equation}
\lim\limits_{z\to \infty} \frac{\ln(e+\norm{z})}{\ln(e+A \norm{z})}= \lim\limits_{z\to \infty} \frac{\ln(e+\norm{z})}{\ln(e+B \norm{z})} = 1~.
\end{equation}
As a consequence there exists $\varepsilon>0$ such that 
\[
\frac{1}{2} < \frac{\ln(e+\norm{z})}{\ln(e+\norm{Sz})} < 2
\] holds for all $z \in \mathbb{R}^{2d}$ satisfying either $\norm{z}< \epsilon$ or $\norm{z}> \epsilon^{-1}$. However, the remaining $z$ constitute a compact set, and $z \mapsto \frac{\ln(e+\norm{z})}{\ln(e+\norm{Sz})}$ is a nonvanishing continuous function, hence there exist nontrivial lower and upper bounds valid on this set. This shows $m_{0,t} \asymp m_{0,t} \circ S$, for all $t$.

Combining our results so far, we obtain for fixed $s,t$ that
\begin{equation}
    \frac{m_{s,t}}{m_{s,t}\circ S}\asymp 1 ~.
\end{equation}

Hence \thref{cor:equiv_weight} together with \thref{thm:classification} provides the conclusion that the operator $\widehat{S}:\MSpace{p,q}{m_{s,t}} \to \MSpace{p,q}{m_{s,t}}$ is bounded iff $p=q$, or $p \not= q$ and $S$ is upper block triangular. Since the latter criterion applies to $S$ iff it applies to its inverse, we further obtain that boundedness of $\widehat{S}$ on $\MSpace{p,q}{m_{s,t}}$ already implies that it is an automorphism. 
\end{example}
We now turn to the second type of polynomial weight functions, where the condition $m \asymp m \circ S$ is actually restrictive.
\begin{example}
According to \cite[Ex. 11.1.1.]{Groechenig2001}, it holds
\begin{equation}
    (1+\norm{z})^s\asymp (1+\norm{x}+\norm{\omega})^s,\qquad x,\omega\in\R^{d},\,z=(x,\omega).
\end{equation}
This weight is symmetric in $(x,\omega)$. However, sometimes, we are only interested in the growth in one direction. This motivates a different class of weight functions, given by 
\begin{equation}
    \begin{split}
            p_{s}(x,\omega)&=p_s(x)=(1+\norm{x})^s, \qquad x,\omega\in\R^{d}, \\
    q_{t}(x,\omega)&=q_t(\omega)=(1+\norm{\omega})^t, \qquad x,\omega\in\R^{d},    
    \end{split}
\end{equation}
with $s,t \in \mathbb{R}$. Since we are only interested in nontrivial weights, we assume $s \not= 0 \not= t$ throughout.  
We use the block-matrix notation 
\begin{equation}
    S = \begin{pmatrix}
A&B\\C&D
\end{pmatrix}\qquad \text{and} \qquad z=\begin{pmatrix}
x\\ \omega
\end{pmatrix},
\end{equation}
and get 
\begin{equation}
\frac{(p_s\circ S)(z)}{p_s(z)} = \left(\frac{1+\norm{Ax+B\omega}}{1+\norm{x}}\right)^{s}
\end{equation}
We first show that $p_s \circ S \asymp p_s$ holds iff $S$ is lower block triangular, i.e., iff $B =0$.

To see the "only-if" part, 
assume $B\neq 0$ and $s>0$. There is $\omega\in\R^{d}$ with $\norm{B\omega}>\norm{Ae_1}$. Therefore,
\begin{equation}
    \left(\frac{1+\norm{Ae_1+B\,n\omega}}{1+\norm{e_1}}\right)^{s} \geq \left(\frac{1+n\norm{B\omega}-\norm{Ae_1}}{2}\right)^{s}\to \infty\quad (n\to\infty).
\end{equation}
In the case $s<0$, we observe that if $S$ is not a lower block triangular matrix, the same applies to $S^{-1}$. Using $z' = S z$, we can use the previous calculation to establish the unboundedness of $\frac{p_s}{p_s\circ S}$. Hence we have $p_s \circ S \not\asymp p_s$ for al $s \not=0$.

Conversely, if $B=0$, $A$ is invertible, and we have the well-known fact that 
\begin{equation}
\frac{p_s\circ S}{p_s}(x,\omega)= \left(\frac{1+\norm{Ax}}{1+\norm{x}}\right)^{s} \asymp 1~.
\end{equation}

This implies that in the case where $S$ is not lower block triangular, neither \thref{thm:multiplication_operator_M_m} nor \thref{thm:main_weighted} is applicable, leaving the question of boundedness of $\widehat{S}$ on $\MSpace{p,q}{p_s}$ wide open. 

In the case where $S$ is lower block triangular, the combination of \thref{cor:equiv_weight} and \thref{thm:classification} yields that $\widehat{S}$ is bounded iff $S$ is upper block triangular, or equivalently, iff $S$ is block-diagonal. 

For the weight $q_t$ with $t \not= 0$, the analogous reasoning as for $p_s$ yields that $q_t \circ S \asymp q_t$ holds iff $S$ is \textit{upper} block triangular. Again, this leaves the question of boundedness of $\widehat{S}$ open for matrices that are not upper block triangular. However, for matrices $S$ that are upper block triangular, \thref{cor:equiv_weight} together with \thref{thm:classification} already implies the boundedness of $\widehat{S}$ on $\MSpace{p,q}{q_t}$, without further restrictions on $S$.  
\end{example}

\section*{Concluding remarks}

Our results concerning the mapping behavior of metaplectic operators $\widehat{S}$ on unweighted modulation spaces give a complete picture. To our knowledge, even the mapping properties of the associated dilation operators $\mathcal{D}_S: L^{p,q}(\mathbb{R}^{2d}) \to L^{p,q}(\mathbb{R}^{2d})$ were not fully clarified prior to our paper, and the parts of Theorem \ref{thm:main_unweighted} dealing with this problem seem to be new. 

The weighted case poses several novel challenges. The first one concerns the choice of weights: Our paper concentrates on even polynomial weights, in order to apply results from \cite{GroechenigToft2011}, and we are not aware of extensions of these results to other weight classes. Note that the arguments of Gr\"ochenig and Toft rely on some fairly subtle properties of algebras of symbol classes, and to us, it seems open whether generalizations of the desired type are available. 

The second unresolved question concerns the mapping properties of metaplectic operators $\widehat{S}:\MSpace{p,q}{m} \to \MSpace{p,q}{m}$ in the case where $\widehat{S}$ is unbounded on the unweighted modulation space $\MSpace{p,q}{}$, and the weights $m$ and $m \circ S^{-1}$ are not equivalent. We presented examples that this case can occur for fairly natural choices of weights. 
Understanding the associated measure-theoretic question, concerning the boundedness properties of the dilation operator $\mathcal{D}_S: L^{p,q}_m(\R^{2d}) \to L^{p,q}_m(\R^{2d})$ clearly constitutes an important step towards answering the question in full. 

\begin{appendix}
\section{Calculation for Equation \texorpdfstring{\eqref{eq:calculation_A_Gauss}}{(3.1)}}\label{app:calculation_A_Gauss}
To prove the unboundedness of $\mathcal{D}_S$ when $S$ is not an upper block triangular matrix, we need the following technical lemma. 
\begin{lemma}\label{lemma:NormFormel}
Let $A=A^T\in \R^{n\times n}$ and $\beta\in\R^{d}$ be given.
It holds
\begin{equation}
\mathcal{I}\coloneqq\int_{\R^{d}} \exp(-\pi x\cdot Ax+2\pi \beta\cdot  x)\ dx = \abs{\det  A}^{-\frac{1}{2}}\ \exp(\pi \beta\cdot  A^{-1}\beta),
\end{equation}
where the integral converges if and only if $A$ is strictly positive definite.
\end{lemma}
\begin{proof}
Let $A=U^T\Lambda U$ with $U$ orthogonal und $\Lambda$ diagonal. Recall that the Lebesgue measure is invariant under the action of the Euclidean group. The integrand is non-negative, so $\mathcal{I}\in[0,\infty]$ is well-defined.
\begin{equation}
\mathcal{I} = 
\int_{\R^{d}} \exp(-\pi (Ux)\cdot \Lambda Ux+2\pi \beta\cdot  U^T Ux)\ dx
= \int_{\R^{d}} \exp(-\pi x\cdot \Lambda x+2\pi (U\beta)\cdot  x)\ dx.
\end{equation}
The integral 
\begin{equation}
\int_{\R}e^{-at^2+bt}dt <\infty
\end{equation}
is finite if and only if $a>0$. Hence $\mathcal{I}<\infty$ if and only if $\Lambda$ is strictly positive  definite.
\begin{equation}
    \begin{split}
        \mathcal{I} &= 
\int_{\R^{d}} \exp\Big(-\pi (\Lambda^{\frac{1}{2}} x)\cdot  (\Lambda^{\frac{1}{2}} x)+2\pi 
(\Lambda^{-\frac{1}{2}}U\beta)\cdot \ \Lambda^{\frac{1}{2}}x \\
&\qquad\qquad -\pi
(\Lambda^{-\frac{1}{2}}U\beta)\cdot \ (\Lambda^{-\frac{1}{2}}U\beta)+\pi 
(\Lambda^{-\frac{1}{2}}U\beta)\cdot \ (\Lambda^{-\frac{1}{2}}U\beta)\Big)\ dx \\
{=}&\abs{\det \ \Lambda^{\frac{1}{2}}}^{-1} \int_{\R^{d}} \exp\Big(-\pi x\cdot x+ 
\pi(\Lambda^{-\frac{1}{2}}U\beta)\cdot \ (\Lambda^{-\frac{1}{2}}U\beta)\Big)\ dx \\
 &= \abs{\det \ U^T\Lambda U}^{-\frac{1}{2}}\ \exp(\pi \beta\cdot  U^T\Lambda^{-\frac{1}{2}}\Lambda^{-\frac{1}{2}} U\beta) 
 = \abs{\det \ A}^{-\frac{1}{2}}\ \exp(\pi \beta\cdot  A^{-1}\beta). 
    \end{split}
\end{equation}
\end{proof}

We now compute the ambiguity function of the Gaussian.
\begin{equation}
    \begin{split}
        &\AF\left(\ g\circ (\varepsilon^2-1)^{\frac{1}{2}}I, g\ \right)(x,\omega)\\
&=e^{\pi i x\cdot \omega}\int_{\R^{d}} \exp\left(-\pi t\cdot (E^2-I)t-\pi(t-x)\cdot (t-x)-2\pi i \omega\cdot  t\right)\ dt \\
&=e^{\pi i x\cdot \omega}\int_{\R^{d}} \exp \left(-\pi t\cdot  E^2t+2\pi x\cdot t-\pi x\cdot  x-2\pi i \omega\cdot  t\right)\ dt \\
=&\;e^{\pi i x\cdot \omega}\abs{\det\ E}^{-1}\int_{\R^{d}} \exp \Big(-\pi t\cdot t+2\pi(E^{-1}x)\cdot t-\pi x\cdot  E^{-2}x  \\
&\qquad\qquad-\pi x\cdot (I-E^{-2})x-2\pi i (E^{-1}\omega)\cdot  (t-E^{-1}x) - 2\pi i (E^{-1}\omega)\cdot  E^{-1}x \Big)\ dt \\[2ex]
\myeq{T_{E^{-1}x}}{=}&\; \abs{\det \ E}^{-1} e^{-\pi x\cdot (I-E^{-2})x - 2\pi i (E^{-1}\omega)\cdot  E^{-1}x+\pi i x\cdot \omega}\cdot
\int_{\R^{d}} \exp \left(-\pi t\cdot t - 2\pi i (E^{-1}\omega)\cdot  t\right)\ dt \\
\myeq{\mathcal{F}g=g}{=}&\; \abs{\det \ E}^{-1} e^{-\pi x\cdot (I-E^{-2})x - 2\pi i (E^{-1}\omega)\cdot  E^{-1}x +\pi i x\cdot \omega- \pi\omega\cdot  E^{-2}\omega} \\
\myeq{\Delta^2+E^{-2}=I}{=}&\; \abs{\det \ E}^{-1} e^{\pi i x\cdot \omega- 2\pi i \omega\cdot  E^{-2}x}e^{-\pi x\cdot \Delta^2 x} e^{- \pi\omega\cdot  E^{-2}\omega}.
    \end{split}
\end{equation}
\section{Calculation for Equation \texorpdfstring{\eqref{eq:calculation_unweighted_norm}}{(3.2)}}\label{app:calculation_unweighted_norm}

We are concerned with the $L^{p,q}$ norm of
\begin{equation}
\mathcal{D}_S f(x,\omega) = \exp \left(-\pi (Ax+B\omega)\cdot \Delta^2(Ax+B\omega)-\pi(Cx+D\omega)\cdot  E^{-2}(Cx+D\omega)\right).
\end{equation}
To this end, we group the scalar products:
\begin{align}
        &\ \quad x\cdot A^T\Delta^2 A x+ 2\omega\cdot B^T\Delta^2 Ax +\omega\cdot B^T\Delta^2 B\omega \\[1.5ex]
&\qquad\qquad + x\cdot C^T E^{-2}C^T x+2\omega\cdot D^T E^{-2} Cx+\omega\cdot D^T E^{-2}D\omega \\[2ex]
&= x\cdot(A^T\Delta^2A +C^T E^{-2}C)x +2\omega\cdot(B^T\Delta^2 A+D^T E^{-2}C)x+\omega\cdot (B^T\Delta^2 B+D^T E^{-2}D)\omega.
\end{align}
For a better overview, we define
\begin{equation}
    \begin{split}
        \Sigma &\coloneqq A^T\Delta^2A +C^T E^{-2}C, \\
    \beta &\coloneqq B^T\Delta^2A+D^T E^{-2}C \\
&\text{and} \\
\Omega&\coloneqq B^T\Delta^2 B+D^T E^{-2}D-\beta\Sigma^{-1}\beta^T.
    \end{split}
\end{equation}
With the new notation, it holds
\begin{equation}
\norm{\mathcal{D}_S f}_{{p,q}} = \norm{\omega\mapsto e^{-\pi\omega\cdot(\Omega+\beta\Sigma^{-1}\beta^T)\omega}\cdot\norm{x\mapsto e^{-\pi x\cdot\Sigma x-2\pi\omega\cdot\beta x}}_p}_q.
\end{equation}
By the definition of $\Sigma$, its spectrum is a subset of $[0,\infty)$. In the cases that we will observe, the spectrum of $\Sigma$ and the spectrum of $\Omega$ will lie in $(0,\infty)$. This is no accident, since $\AF( \varphi,g)\in \mathcal{S}(\R^{d})$ for all $\varphi\in \mathcal{S}(\R^{d})\subset \MSpace{p,q}{}$. If $p<\infty$, the inner integral equals
\begin{equation}\begin{split}
n(\omega)&=\norm{x\mapsto e^{-p\pi x\cdot\Sigma x-2p\pi\omega\cdot\beta x}}_1^{\frac{1}{p}} \myeq{\ref{lemma:NormFormel}}{=} \abs{\det \ p\Sigma}^{-\frac{1}{2p}}\ \left(e^{p\pi(\omega\cdot\beta)\Sigma^{-1}\beta^T\omega}\right)^{\frac{1}{p}}\\
&{\asymp} \abs{\det \ \Sigma}^{-\frac{1}{2p}}e^{\pi\omega\cdot\beta\Sigma^{-1}\beta^T\omega}.
\end{split}\end{equation}
\\[1ex]
If $p=\infty$, then the function
\begin{equation}
    F(x,\omega)\coloneqq e^{-\pi x\cdot\Sigma x-2\pi\omega\cdot\beta x} = e^{-\pi (x+\Sigma^{-1}\beta^T\omega)\cdot\Sigma (x+\Sigma^{-1}\beta^T\omega) +\pi \omega\cdot \beta\Sigma^{-1} \beta^T\, \omega}\leq e^{\pi \omega\cdot \beta\Sigma^{-1} \beta^T\, \omega} ~,
\end{equation}
with equality if and only if $x = - \Sigma^{-1}\beta^T\, \omega$.
Hence, for all $p\in[1,\infty]$ holds
\begin{equation}\begin{split}
\norm{\mathcal{D}_S f}_{{p,q}} &\asymp \abs{\det \ \Sigma}^{-\frac{1}{2p}}\norm{\omega\mapsto
 e^{-\pi\omega\cdot(\Omega+\beta\Sigma^{-1}\beta^T)\omega}\cdot e^{\pi\omega\cdot \beta\Sigma^{-1}\beta^T\omega}}_q \\
&= \abs{\det \ \Sigma}^{-\frac{1}{2p}}\norm{\omega\mapsto e^{-\pi\omega\cdot \Omega\omega}}_q.
\end{split}\end{equation}
For the case $q=\infty$, it immediately follows that the remaining norm is $1$. For all other $q\in [1,\infty)$ holds
\begin{equation}
\norm{\mathcal{D}_S f}_{{p,q}} 
\asymp \abs{\det \ \Sigma}^{-\frac{1}{2p}}\ \abs{\det (q\ \Omega)}^{-\frac{1}{2q}}\asymp \abs{\det \ \Sigma}^{-\frac{1}{2p}} \abs{\det \ \Omega}^{-\frac{1}{2q}}.
\end{equation}
All in all, 
\begin{equation}
    \norm{\mathcal{D}_S f}_{{p,q}} \asymp \abs{\det \ \Sigma}^{-\frac{1}{2p}} \abs{\det \ \Omega}^{-\frac{1}{2q}}
\end{equation}
for all $p,q\in [1,\infty]$.
\end{appendix}
\nocite{CauliEtAl2019}
\bibliographystyle{acm}

\bibliography{mybib}

\end{document}